\documentclass{amsart}
\usepackage{amssymb}
\usepackage[mathscr]{eucal}
\usepackage{eufrak}
\usepackage{xspace}
\usepackage{mathtools}
\usepackage{color}

\newcommand{\R}{\mathbb{R}}

\newcommand{\Z}{\mathbb{Z}}
\newcommand{\N}{\mathbb{N}}
\newcommand{\C}{\mathbb{C}}
\newcommand{\Q}{\mathbb{Q}}

\newcommand{\LL}{\mathcal{L}}

\newcommand{\F}{{\mathcal F}}

\newcommand{\po}{\partial}

\newcommand{\wto}{\rightharpoonup}
\newcommand{\ve}{\varepsilon}

\newcommand{\la}{\langle}
\newcommand{\ra}{\rangle}

\newcommand{\loc}{{\text{\rm loc}}}
\newcommand{\X}{\times}

\renewcommand{\d}{\delta}
\renewcommand{\l}{\lambda}

\renewcommand{\a}{\alpha}
\renewcommand{\b}{\beta}

\newcommand{\s}{\sigma}
\newcommand{\g}{\gamma}

\renewcommand{\k}{\kappa}
\newcommand{\sgn}{\text{\rm sgn}}

\newcommand{\Om}{\Omega}
\newcommand{\om}{\omega}

\newcommand{\supp}{\text{\rm supp}\,}

\newcommand{\M}{{\mathcal M}}

\renewcommand{\div}{\text{\rm div}\,}

\renewcommand{\supp}{\text{\rm supp}\,}

\newcommand{\cO}{{\mathcal O}}

\newcommand{\bff}{{\mathbf f}}

\newcommand{\abf}{\mathbf{a}}

\newcommand{\Lip}{\text{\rm Lip}}

\newcommand{\B}{{\mathcal B}}

\newcommand{\bbE}{{\mathbb E}}
\newcommand{\bbT}{{\mathbb T}}
\newcommand{\bbP}{{\mathbb P}}
\newcommand{\bbG}{{\mathbb G}}

\newcommand{\BAP}{\operatorname{BAP}}
\newcommand{\AP}{\operatorname{AP}}

\newcommand{\BUC}{\operatorname{BUC}}

\newcommand{\dist}{\text{dist\,}}
\newcommand{\diam}{\text{diam\,}}
\renewcommand{\Lip}{\text{Lip\,}}

\newcommand{\ff}{\mathfrak{f}}

\newcommand{\mm}{\mathfrak{m}}

\newcommand{\G}{{\mathcal G}}

\newcommand{\Me}{\operatorname{M}}
\newcommand{\Sp}{\operatorname{Sp}}
\newcommand{\Gr}{\operatorname{Gr}}

\newcommand{\Medint}{\mkern12mu\mbox{\vrule height4pt
         depth-3.2pt
          width5pt}\mkern-16.5mu\int\nolimits}

\theoremstyle{plain}
\newtheorem{theorem}{Theorem}[section]
\newtheorem{corollary}{Corollary}[section]

\newtheorem{lemma}{Lemma}[section]
\newtheorem{proposition}{Proposition}[section]
\theoremstyle{definition}
\newtheorem{definition}{Definition}[section]
\theoremstyle{remark}

\numberwithin{equation}{section}

\begin{document}

\title[Invariants measures for stochastic conservation laws]
{Invariant measures for stochastic\\ conservation laws with Lipschitz flux \\ in the space of 
 almost periodic functions}

\author[C. Espitia]{Claudia Espitia}
\thanks{Claudia Espitia thankfully acknowledges the support from CNPq, through grant proc.\ 140268/2019-7}
\address{Instituto de Matem\'atica Pura e Aplicada - IMPA\\
         Estrada Dona Castorina, 110 \\
         Rio de Janeiro, RJ 22460-320, Brazil}
\email{claudia.duarte@impa.br}

\author[H.~Frid]{Hermano Frid}
\thanks{H.~Frid gratefully acknowledges the support from CNPq, through grant proc.\ 305097/2019-9, and FAPERJ, through grant proc.\ E-26/200.958/2021.}
\address{Instituto de Matem\'atica Pura e Aplicada - IMPA\\
         Estrada Dona Castorina, 110 \\
         Rio de Janeiro, RJ 22460-320, Brazil}
\email{hermano@impa.br}

\author[D.~Marroquin]{Daniel Marroquin}
\thanks{D.~Marroquin thankfully acknowledges the support from CNPq, through grant proc. 150118/2018-0.}

\address{Instituto de Matem\'{a}tica - Universidade Federal do Rio de Janeiro\\
Cidade Universit\'{a}ria, 21945-970, Rio de Janeiro, Brazil}
\email{marroquin@im.ufrj.br}

\date{}

\begin{abstract} We study the well-posedness of the initial value problem and the long-time behaviour of almost periodic solutions to
stochastic scalar conservation laws in any space dimension, under the assumption 
of Lipschitz continuity of the flux functions and a non-degeneracy condition. We show the existence and uniqueness of an invariant measure in a separable subspace of the space of  Besicovitch almost periodic functions.
\end{abstract}

\maketitle

\section{Introduction}\label{S:1}

We study the existence and uniqueness of  invariant measures for  first-order scalar conservation 
laws with stochastic forcing in the setting of almost periodic functions. So, we consider the initial value problem 
\begin{align}
&du+\div \bff(u)\,dt=\Phi\,dW(t), \quad x\in\bbG_N, \ t\in(0,T), \label{e1.1}\\
& u(0,x)=u_0(x),\quad x\in\bbG_N,  \label{e1.2}
\end{align}
where  $\bff \in C_b^2(\R;\R^N)$ and  $\bbG_N$ denotes the Bohr compact associated with  
 the space of the almost periodic functions in $\R^N$, denoted by $\AP(\R^N)$, through a classical extension of the  Stone-Weierstrass theorem (see, e.g,  \cite{DS}, p.274--276, Theorem~18 and Corollary~19),
 whose detailed characterization and main properties  we explain later on. By $C_b^k$ we mean $k$-times differentiable with bounded continuous derivatives up to $k$-th order.  In particular, $\bff$ is Lipschitz continuous. 
%

The purpose of this paper is to extend to the context of the Besicovitch almost periodic functions the result on  the existence and uniqueness of invariant measures for periodic solutions of \eqref{e1.1} by Debussche and Vovelle in \cite{DV2}, at least in the case of Lipschitz continuous flux functions. As in \cite{DV2} we consider an additive noise $\Phi\,dW(t)$, whose detailed definition we recall further below.

More precisely, here, in the first part of the paper, we establish the well-posedness of \eqref{e1.1}-\eqref{e1.2} in $L^1(\bbG_N)$. This is achieved
after establishing the well-posedness  of the problem \eqref{e1.1}-\eqref{e1.2}, considered as a Cauchy problem in $\R^N$ with initial function 
in $\B^1(\R^N)\cap L^\infty(\R^N)$. Here, $\B^1(\R^N)$ is the space of Besicovitch almost periodic functions, which is isometrically isomorphic with $L^1(\bbG_N)$ as we describe with more details in the next section.  In sum,  $\AP(\R^N)$ is isometrically isomorphic with $C(\bbG_N)$, and this isomorphism, together with the mean-value property of the functions in $\AP(\R^N)$, yields the isometric isomorphism between $\B^1(\R^N)$ and $L^1(\bbG_N)$.    

Then, we consider a closed subspace of $L^1(\bbG_N)$, denoted by $L^1(\bbG_{*N})$, roughly defined as follows. 
Let $\AP_*(\R^N)$ be a closed subspace of $\AP(\R^N)$ whose functions have spectrum contained in a finitely  generated additive subgroup of $\R^N$. Then $\AP_*(\R^N)$ is isometrically isomorphic to $C(\bbG_{*N})$, where $\bbG_{*N}$ is the compactification of $\R^N$ generated by the algebra $\AP_*(\R^N)$, by the same extension of the Stone-Weierstrass theorem referred to above. 
Clearly, we have the isometric embedding $C(\bbG_{*N})\hookrightarrow C(\bbG_N)$. 
$L^1(\bbG_{*N})$ is then the closure in $L^1(\bbG_N)$ of  $C(\bbG_{*N})$. We then extend to $L^1(\bbG_{*N})$  the result of \cite{DV2} on the existence and uniqueness of invariant measure for \eqref{e1.1},  in the case where $\bff\in C_b^2(\R)$.

 The subject of the asymptotic behaviour of oscillatory solutions of deterministic conservation laws has a very long history that goes back to the first papers on scalar conservation laws (see, e.g., \cite{Ho, Ol, Lx}).  With the introduction of new compactness frameworks such as compensated compactness in, e.g., \cite{Ta, DP1, DP2}, kinetic formulation and averaging lemmas in, e.g., \cite{LPT}, this research gained a great impulse  (see, e.g., \cite{CF0,Fr0,CP, Pv0,Pv, FL, Pv2, GS, GS2}, among others). We also mention the elegant approach in \cite{Da} based on infinite dimension dynamical systems ideas. On the other hand, in the context of stochastic scalar conservation laws, the study of the asymptotic behaviour of periodic solutions was inaugurated with \cite{EKMS} for the Burgers equation, based on  infinite dimension dynamical systems ideas, which here seems to be the appropriate  approach. The result in \cite{EKMS} was extended to more general conservation laws in several space dimensions in \cite{DV2}. We refer to the latter for more references on the subject of invariant measures for stochastic conservation laws and other correlated stochastic partial differential equations. Also see \cite{DZ} for a general account on the basic concepts of infinite dimensional dynamical systems associated with stochastic equations. 
  The decay of periodic solutions of scalar conservation laws with rough flux in several space dimensions has been addressed in \cite{GS, GS2}; we do not address this type of stochastic equations here.  
 
 The theory of stochastic partial differential 
equations has had intense progress in the 
last three decades. We cite the 
treatise \cite{DPZ} for a basic 
general account of this theory and references.
More specifically, concerning the theory 
of stochastic conservation laws, we 
mention the first contributions by Kim \cite{KJ}, 
and Feng and Nualart \cite{FN}. The latter was further 
developed in Chen, Ding, and Karlsen in \cite{CDK}.
An inflection in the course of this theory 
was achieved by Debussche and Vovelle \cite{DV} 
with the introduction of the notion of stochastic 
kinetic solution, extending 
the corresponding deterministic concept introduced by 
Lions, Perthame, and Tadmor \cite{LPT}.  
We also mention the independent development 
in this theory made by Bauzet, Vallet, and Wittbold \cite{BVW1}.

The plan of this paper is as  follows. In Section~\ref{S:1'} we setup the framework on which we will establish our results. In Section~\ref{S:2} and Section~\ref{S:3}, we start by establishing the well-posedness of the Cauchy problem in $\R^N$ of Besicovitch almost periodic ($\BAP$) entropy solutions. Then, in Section~\ref{S:4}, we establish the existence of entropy solutions in $\bbG_N$ for \eqref{e1.1}-\eqref{e1.2} and introduce the notion of semigroup solutions for which the contraction property in $L^1(\bbG_N)$ holds. In Section~\ref{S:5} we discuss the method of reduction to the periodic case, originally introduced in \cite{Pv}, restricting the well-posedness analysis to $L^1(\bbG_{*N})$, establishing a perfect correspondence between $L^1(\bbG_{*N})$-semigroup solutions and entropy periodic solutions in $L^1(\bbT^P)$, where $P$ is the cardinality of the generators the additive subgroup of $\R^N$ containing the spectra of the functions in $C(\bbG_{*N})$. Finally, in Section~\ref{S:6}, we establish the existence and uniqueness of an invariant measure for \eqref{e1.1} in $L^1(\bbG_{*N})$.    

\section{Preliminaries, Notations  and Main Results} \label{S:1'}

We recall that the space of real-valued almost periodic functions in $\R^N$, $\AP(\R^N)$, was introduced by Bohr (see, e.g.,  \cite{B}) who also characterized it as the closure in $C_b(\R^N)$, with respect to the $\sup$-norm, of the finite linear combinations of the trigonometric functions $\cos 2\pi\l\cdot x$ and $\sin 2\pi\l\cdot x$, $\l \in\R^N$, or, equivalently, the real part of the closure in $C_b(\R^N;\C)$ of the complex space spanned by  $\{ e^{i2\pi\l\cdot x}\,:\, \l\in\R^N\}$. 
A well known theorem by Bochner (see, e.g.,  \cite{B}) 
states that a function  $g\in C_b(\R^N)$ belongs to
$\AP(\R^N)$   if and only if the set $\{g(\cdot+\l)\,:\, \l\in\R^N\}$ is contained in a compact of $C_b(\R^N)$. Also, it  is a sub-algebra of the space of bounded uniformly continuous functions $\BUC(\R^N)$, whose elements
$g$ possess a mean value $\Me(g)$, which can be defined by
$$
\Me(g):= \lim_{R\to\infty} R^{-N}\int_{C_R}g(x)\,dx,
$$
where
\begin{equation}\label{e1.CR}
C_R:=\{x=(x_1,\cdots,x_N)\in\R^N\,:\, |x_i|\le R/2\}.
\end{equation}
More precisely,  the mean-value $\Me(g)$ is the limit in the weak-star topology of $L^\infty(\R^N)$ of $g(\ve^{-1}\,\cdot)$ as $\ve\to0$, which always exists for functions in $\AP(\R^N)$.  
Moreover, it holds that $\Me(g)=\Me(g(\cdot+\l))$, uniformly in $\l\in\R^N$. We also use the notation
$$
\Me(g)=\Medint_{\R^N} g(x)\,dx.
$$
Given $g\in\AP(\R^N)$, the spectrum of $g$, $\Sp(g)$, is defined by
$$
\Sp(g):=\{\l\in\R^N\,:\, a_\l:=\Medint_{\R^N} e^{-i2\pi \l\cdot x}g(x)\,dx\ne0\}.
$$   
 A basic fact on $\Sp(g)$ is that it is a countable set, which follows easily from Bessel's inequality when we introduce in $\AP(\R^N)$ the inner product 
 $$
 \la g,h\ra:=\Me(gh).
 $$
For $g\in\AP(\R^N)$, we denote by $\Gr(g)$ the smallest additive group containing $\Sp(g)$.

It is also a well known fact  that $\AP(\R^N)$ is isometrically isomorphic with the space $C(\bbG_N)$, where 
$\bbG_N$ is the so called Bohr compact which is a compact topological group (see, e.g., \cite{DS,Loo}). More specifically,  we may view $\R^N$ as densely embedded in $\bbG_N$ and 
the functions in $\AP(\R^N)$ may be continuously extended to $\bbG_N$ and these extensions form the space  $C(\bbG_N)$. The group operation $+:\bbG_N\X\bbG_N\to\bbG_N$ is the extension to $\bbG_N$ of the addition operation in 
$\R^N$, $+:\R^N\X\R^N\to\R^N$, which is  continuous also in the topology induced by the embedding  of  $\R^N$ into $\bbG_N$ (see, e.g.,  \cite{DS2}, Theorem~XI.2). The Haar measure on $\bbG_N$, $\mm$, such that $\mm(\bbG_N)=1$, is the measure induced by the  mean value $g\mapsto \Me(g)$ defined for all $g\in\AP(\R^N)\sim C(\bbG_N)$.  
The topology 
in $\bbG_N$ is generated by the images by the referred isomorphism  (also called Gelfand transforms) of the functions $e^{-i2\pi\l\cdot x}$. Here and elsewhere in what follows, although we are mainly dealing with real functions,  we switch freely between the real and  the complex version of $\AP(\R^N)$ whenever we want to take advantage of the fact that the latter is generated by the complex exponentials  $e^{-i2\pi\l\cdot x}$. The translations $\tau_y:\R^n\to\R^N$, $\tau_yx=x+y$,
$y\in\R^N$, extend as homeomorphisms $\tau_y:\bbG_N\to\bbG_N$. Therefore, we can define directional derivatives $D_y g(x)$ of functions  $g\in C(\bbG_N)$ at a point $x\in\bbG_N$, for $y\in\R^N$, $|y|=1$,  by the usual formula, $D_yg(x):=\lim_{|h|\to0} (g(x+hy)-g(x))/h$ whenever the limit exists. In particular, when $y=e_i$, where $e_i$ is the $i$-th element of the canonical basis,
we get the partial derivatives $D_i g(x)$ or $\po_{x_i}g(x)$, or yet $\po g(x)/\po x_i$, $i=1,\cdots,N$.  We then denote by $C^\ell(\bbG_N)$ the space of functions in $C(\bbG_N)$ whose derivatives up to order $\ell$ also belong to $C(\bbG_N)$, for $\ell\in\N\cup\{0,\infty\}$. It is easy to see that $C^\ell(\bbG_N)$ is isometrically isomorphic with $\AP^\ell(\R^N)$, where the latter is the subspace of $\AP(\R^N)$ whose derivatives up to the order $\ell$ are in $\AP(\R^N)$, for  $\ell\in\N\cup\{0,\infty\}$.

As usual, if  $(\Om, \F, \bbP, (\F_t)_{t\ge0}, (\b_k(t))_{k\in\N})$ is a stochastic basis,  $W$ is a cylindrical Wiener process, $W=\sum_{k\ge1}\b_k e_k$, where $\b_k$ are independent Brownian processes, $(\F_t)_{t\ge0}$  is the completion of the canonical filtration associated with $(\b_k(t))_{k\ge1}$, and $\{e_k\}_{k\ge1}$ is a complete orthonormal system in a Hilbert space $H$.  The map $\Phi: H\to L^2(\bbG_N)$ is defined by $\Phi e_k=g_k$ where $g_k\in C^2(\bbG_N)\sim \AP^2(\R^N)$. We assume that there exists a sequence of positive numbers $(\a_k)_{k\ge1}$ satisfying $D_0:=\sum_{k\ge1}(\a_k+\a_k^2)<\infty$ such that 
\begin{equation}\label{e1.3-0}
|g_k(x)|+|\nabla_x g_k(x)|+|\nabla_x^2 g_k(x)|\le \a_k,\quad \forall x\in\R^N.
\end{equation}
Observe that from\eqref{e1.3-0} it follows 
\begin{align}
& G^2(x)=\sum_{k\ge1}|g_k(x)|^2\le D_0, \label{e1.3}\\
&\sum_{k\ge1}|g_k(x )-g_k(y)|^2\le D_0 |x-y|^2, \label{e1.4}
\end{align}
for all $x,y\in \R^N$. 

 Existence and uniqueness of an invariant measure for \eqref{e1.1} in the periodic case is proved in \cite{DV2} based on a well-posedness theory for \eqref{e1.1}-\eqref{e1.2} in $L^1(\bbT^N)$ which extends to $L^1(\bbT^N)$ the theory developed in \cite{DV} (see also \cite{DV'}). 
 From the latter it follows that given two 
initial data $u_0^1$ and $u_0^2\in L^1(\bbT^N)$, the following contraction property holds (see theorem~5 in \cite{DV2}) :
\begin{equation}\label{e1.4P}
\|u^1(t)-u^2(t)\|_{L^1(\bbT^N)}\le \|u_0^1-u_0^2\|_{L^1(\bbT^N)}, \ a.s.,
\end{equation}
which in turn allows the definition of a transition semigroup in $L^1(\bbT^N)$:
$$
P_t\phi(u_0)=\bbE(\phi(u(t))),\quad \phi\in \B_b(L^1(\bbT^N)).
$$

Here we establish a well-posedness theory for \eqref{e1.1}-\eqref{e1.2} in the space of Besicovitch almost periodic functions,  when the flux function $\bff(u)$ is Lipschitz continuous. Such well-posedness theory is ultimately extended to $L^1(\bbG_N)$. More specifically, we prove the following theorem, for which we also assume a non-degeneracy condition on \eqref{e1.1} (see \eqref{e1.GH}).
 
\begin{theorem}\label{T:3new0} Assume \eqref{e1.GH} holds. Given, $u_0\in L^1(\bbG_N)$, there exists a $L^1(\bbG_N)$-entropy solution of  \eqref{e1.1}-\eqref{e1.2} in the sense of Definition~\ref{D:4.1} which is the  $L^1(\Om; \\ L^\infty((0,T) ;L^1(\bbG_N)))$ limit of
$\BAP$-entropy solutions of \eqref{e2.1}-\eqref{e2.2} with initial data in $\AP(\R^N)$. We call the latter a $L^1(\bbG_N)$-semi-group solution. 
Moreover,
given two 
initial data $u_0^1$ and $u_0^2\in L^1(\bbG_N)$, the following contraction property holds for the corresponding
 $L^1(\bbG_N)$-semigroup solutions. 
 \begin{equation}\label{e1.4AP0}
\|u^1(t)-u^2(t)\|_{L^1(\bbG_N)}\le \|u_0^1-u_0^2\|_{L^1(\bbG_N)}, \ a.s.
\end{equation}
\end{theorem}

%

In order to address the existence and uniqueness of invariant measures, we restrict our analysis to a subspace of $L^1(\bbG_N)$ whose elements possess spectra contained in a fixed finitely generated additive subgroup of $\R^N$.  
 Namely, for a certain finite set $\Lambda:=\{\l_1,\cdots, \l_P\}\subset\R^N$, linearly independent over $\Z$, let $C(\bbG_{*N})$ be the closed sub-algebra of  $C(\bbG_N)$ formed by the functions $g\in C(\bbG_N)$ such that $\Sp(g)$ is contained in the smallest additive group containing $\Lambda$, that is, the $\Z$-linear space generated by $\Lambda$.  $L^1(\bbG_{*N})$ is then the closure in $L^1(\bbG_N)$ of $C(\bbG_{*N})$.  We  prove that, when the noise functions $g_k$, $k\in\N$, belong to $C(\bbG_{*N})$ for $u_0\in L^1(\bbG_{*N})$, the $L^1(\bbG_N)$-semigroup solution given by Theorem~\ref{T:3new0}  gives rise to a transition semigroup in $L^1(\bbG_{*N})$ for \eqref{e1.1}, defined as in the periodic case:
  \begin{equation}\label{e1.trans}
P_t\phi(u_0)=\bbE(\phi(u(t))),\quad \phi\in \B_b(L^1(\bbG_{*N})).
\end{equation}

We then prove the existence and uniqueness of an invariant measure for \eqref{e1.1} on $L^1(\bbG_{*N})$, provided  the noise functions are in $C(\bbG_{*N})$, and, as in \cite{DV2}, besides the condition \eqref{e1.3-0}, satisfy
\begin{equation}\label{e1.zeromean}
\Medint_{\R^N}g_k(x)\,dx=0,\quad \forall k\in\N.
\end{equation}
For that, we also need a further non-degeneracy condition (see \eqref{e5.NDC}).

\begin{theorem}\label{T:1.1}  Assume condition \eqref{e5.NDC} holds, with 
$\iota^\vartheta(\d)$ defined by \eqref{e5.NDC1}.
 Then there is a unique  invariant measure for the transition semigroup $P_t$ in 
 $\B_b(L^1(\bbG_{*N}))$.
 \end{theorem}

We recall that the existence of an invariant measure in the periodic setting in \cite{DV2} is obtained
as a consequence of the regularity estimate ({\em cf.} (41) in \cite{DV2}) 
$$
\bbE\|u\|_{L^1(0,T; W^{s,q}(\bbT^N)}\le k_0(\bbE\|u_0\|_{L^3(\bbT^N)}^3+1+T),
$$
for some $\k_0>0$, and an application of the Krylov-Bogoliubov theorem (see, e.g., \cite{DZ}). The uniqueness result in \cite{DV2} is also a consequence of this regularity estimate  translated to $[t,t+T]$, for any $t\ge0$, applied to two solutions $u^1, u^2$, corresponding to initial data  $u_0^1, u_0^2\in L^3(\bbT^N)$ ({\em cf.} (42) in \cite{DV2})
\begin{multline*}
\bbE\left(\int_t^{t+T}\|u^1(s)\|_{L^1(\bbT^N)}+\|u^2(s)\|_{L^1(\bbT^N)}\,ds\,|\,\F_t\right)\\
\le \k_1(\|u^1(t)\|_{L^3(\bbT^N)}^3+\|u^2(t)\|_{L^3(\bbT^N)}^3+1+T),
\end{multline*}
 with which it is shown in \cite{DV2} that one can define the following succession of almost surely finite stopping times,    
 $\{\tau_\ell\}_{\ell\in\N\cup\{0\}}$, with $\tau_0=0$:
 $$
 \tau_\ell:=\inf\{t\ge \tau_{\ell-1}+T\,:\, \|u^1(t)\|_{L^1(\bbT^N)}+\|u^2(t)\|_{L^1(\bbT^N)}\le 2\k_1\}.
 $$
 A crucial proposition is then proved in \cite{DV2} ({\em cf.} Proposition~12 in \cite{DV2})  showing that, for ${\bf a}(\xi):={\bf f}'(\xi)$, assuming ${\bf a}'(\xi)$ uniformly bounded in $\R$, then for any $\ve>0$, there exist $T>0$ and $\eta>0$ such that
 if 
 $$
 \|u(0)\|_{L^1(\bbT^N)}\le 2\k_1\quad\text{and}\quad \sup_{t\in[0,T]}\|\sum_{k\ge1}g_k(x)\b_k(t)\|_{W^{1,\infty}(\bbT^N)}\le \eta,
 $$ 
then 
 $$
 \frac1T\int_0^T\|u(s)\|_{L^1(\bbT^N)}\,ds\le \frac{\ve}2.
 $$
 This then implies, after some manipulations, that the sequence 
$\frac1T\int_{\tau_\ell}^{\tau_\ell+T}\|u^1(t)-u^2(t)\|_{L^1(\bbT^N)}$ converges, as $\ell\to\infty$, to 0 in probability, which implies, since $t\mapsto \|u^1(t)-u^2(t)\|_{L^1(\bbT^N)}$ is non-increasing, that 
$$
\lim_{t\to\infty}\|u^1(t)-u^2(t)\|_{L^1(\bbT^N)}=0,
$$
which finally implies the uniqueness of the invariant measure, as explained in Section~\ref{S:6}.
 

\section{Almost periodic solutions}\label{S:2}

We begin by studying the initial value problem corresponding to \eqref{e1.1}-\eqref{e1.2} in 
$\R^N$,
\begin{align}
&du+\div \bff(u)\,dt=\Phi\,dW(t), \quad x\in\R^N, \ t\in(0,T),\label{e2.1}\\
&u(0,x)=u_0(x),\quad x\in\R^N.\label{e2.2}
\end{align}

Let us define
$$
J(x,t):=\sum_{k=1}^\infty g_k(x) \b_k(t),
$$
and
$$
w=u-J.
$$
As in \cite{Pv}, let us define
 \begin{equation}\label{e2.Np}
 N_p(u)=\limsup_{R\to\infty}\left(R^{-N}\int_{C_R}|u(x)|^p\,dx\right)^{1/p},\quad p\ge1,
 \end{equation}
 with $C_R$ as in \eqref{e1.CR}.
 We observe that $N_p$ is a norm when restricted to $\AP(\R^N)$. Let $\B^1(\R^N)$ denote the completion of $\AP(\R^N)$ with respect to the norm $N_1$. For a function  $u\in \B^1(\R^N)$ we have
\begin{equation}\label{e2.N1}
N_1(u)=\Medint_{\R^N}|u(x)|\,dx.
\end{equation}

\begin{definition}\label{D:2.1} Let $u_0\in L^\infty\cap \B^1(\R^N)$ and  $T>0$ be given. A $L_\loc^1\cap \B^1(\R^N)$-valued stochastic process,  adapted
to $\{\F_t\}$, is said to be a Besicovitch almost periodic entropy solution ($\BAP$-entropy solution) of \eqref{e2.1}-\eqref{e2.2} if, for almost all $\om\in\Om$,
\begin{enumerate}
 \item $u$ is $L_\loc^1\cap \B^1(\R^N)$-weakly continuous on $[0,T]$,
 \item $u\in L^\infty([0,T]; L_\loc^1\cap\B^1(\R^N))$,  
 \item for all nonnegative $\varphi\in C_0^\infty ((-T,T)\X\R^N)$ and $\a\in\R$
 \begin{multline}\label{e2.3}
 \int_0^T\int_{\R^N} \{(|w-\a|-|u_0-\a|)\varphi_t+\sgn(w-\a)(\bff(w+J)-\bff(\a+J))\cdot\nabla\varphi\\
 -\sgn(w-\a)\bff'(\a+J)\cdot\nabla J\varphi\}\,dx\,dt\ge0.
 \end{multline}
 \end{enumerate}
 \end{definition}
 
 The following result is going to be proved  in the next section, where we  also state the non-degeneracy 
 condition which we need to assume on \eqref{e1.1}.  

\begin{theorem}~\label{T:2.1}  Assume \eqref{e1.GH} holds.  Given $T>0$, there is a $\BAP$-entropy solution of \eqref{e2.1}-\eqref{e2.2}. Furthermore, the solution is pathwise unique.
\end{theorem} 
 
  We next establish the fundamental contraction property for $\BAP$-entropy solutions of  \eqref{e2.1}-\eqref{e2.2}. 
 \begin{proposition}[$L^1$-mean contraction property]~\label{P:2.1} Let $u(t,x), v(t,x)\in L^\infty([0,T];\\ \B^1(\R^N))$  be two $\BAP$-entropy solutions of \eqref{e2.1}-\eqref{e2.2} with initial data $u_0(x), v_0(x)$. 
 Then,  a.s., for a.e.\ $t>0$ 
 \begin{equation}\label{e2.4} 
\Medint_{\R^N}|u(t,x)- v(t,x)|\,dx \le \Medint_{\R^N}|u_0(x)-v_0(x)|\,dx.
 \end{equation}
 \end{proposition}
 
 \begin{proof} Let $w=u-J$ and $\hat w=v-J$. We begin by recalling the fundamental Kruzhkov-type inequality as established in \cite{KJ} (see equation (3.20) in \cite{KJ}, its extension to the multi-dimensional case is straightforward) 
 \begin{multline}\label{e2.5}
 \int_0^T\int_{\R^N} \{|w(t,x)-\hat w(t,x)|\varphi_t\\
 +\sgn(w-\hat w)(\bff(w+J)-\bff(\hat w+J))\cdot\nabla\varphi\}\,dx\,dt\ge0,
 \end{multline}  
 for all nonnegative $\varphi\in C_0^\infty ((0,T)\X\R^N)$. As in \cite{Pv}, we choose $\varphi$ in \eqref{e2.5} of the
 form $\varphi=R^{-N}\chi_\nu(t) g(x/R)$, with $g\in C_c^\infty(\R^N)$ such that $0\le g(y)\le 1$, $g(y)\equiv1$ in the cube $C_1$, $g(y)=0$ in the complement of the cube $C_k$, $k>1$; also $\chi_\nu$ is a smooth approximation of 
 the indicator function of the interval $(t_0,t_1]$. We define the set
 $$
 F=\{t>0\,:\, (t,x)\ \text{is a Lebesgue point of $|w(t,x)-\hat w(t,x)|$ for a.e.\ $x\in\R^N$}\}.
 $$
 Then, using Fubini's theorem and density as in lemma~1.2 of \cite{Pv},  we deduce that $F\subset\R_+$ is a set of full measure and each $t\in F$ is a Lebesgue point of the functions 
 $$
 I_R(t)=R^{-N}\int_{\R^N}|w(t,x)-\hat w(t,x)| g(x/R)\,dx
 $$
 for all $R>0$ and all $g(y)\in C_c(\R^N)$. Then, for $t_0,t_1\in F$ we get
 \begin{multline*}
 I_R(t)\le I_R(t_0) \\+ R^{-N-1}\int_{(t_0,t_1)\X\R^N} \sgn(w-\hat w)(\bff(w+J)-\bff(\hat w+J))\cdot\nabla_y g(x/R)\,dt\,dx.
 \end{multline*}
 Then, since from Definition~\ref{D:2.1} it follows that $w(t,\cdot),\hat w(t,\cdot)$ converge as $t\to0$ in $L_\loc^1$ to $w_0(x),\hat w_0(x)$, we get  for all $t\in F$
 \begin{multline}\label{e2.6} 
 I_R(t)\le I_R(0)\\
 + R^{-N-1}\int_{(0,t)\X\R^N} \sgn(w-\hat w)(\bff(w+J)-\bff(\hat w+J))\cdot\nabla_y g(x/R)\,dt\,dx.
 \end{multline}
 Now,  we have
\begin{multline}\label{e2.7} 
R^{-N-1}\left|\int_{(0,t)\X\R^N} \sgn(w-\hat w)(\bff(w+J)-\bff(\hat w+J))\cdot\nabla_y g(x/R)\,dt\,dx
\right|\\
\le \Lip(\bff) \|\nabla g\|_\infty k^N R^{-1}\frac1{(kR)^N}\int_{(0,t)\X C_{Rk}}|w(s,x)-\hat w(s,x)|\,dt\,dx\\
\to 0, \quad\text{as $R\to\infty$.}
\end{multline}
Using the properties of $g$ we easily deduce that, for all $t\ge0$, 
$$
N_1(w(t,\cdot)-\hat w(t,\cdot))\le \limsup_{R\to\infty} I_R(t) \le  k^N N_1(w(t,\cdot)-\hat w(t, \cdot)),
$$
for any $k>1$, from which together with \eqref{e2.6} and \eqref{e2.7}, it follows, making $R\to\infty$, that
$$
N_1(w(t,\cdot)-\hat w(t,\cdot))\le k^N N_1(w_0-\hat w_0),
$$
for all $t\in F$, and then making $k\to 1$ we arrive at \eqref{e2.4}.   

 \end{proof}

From \eqref{e2.5}, we also get the uniqueness of entropy solutions to \eqref{e2.1}-\eqref{e2.2} with the stability in $L_\loc^1$, which we establish in the following theorem. Let $\rho_*$ be given by
\begin{equation}\label{e3.rho}
\rho_*(x):=C_N e^{-\sqrt{1+|x|^2}}, \quad \text{with $C_N$ such that $\int_{\R^N}\rho_*(x)\,dx=1$}.
\end{equation}
Observe that $\rho_*$ has the property that 
\begin{equation}\label{e2.200}
|D_x^\a\rho_*(x)|\le C_\a\rho_*(x)
\end{equation}
 for all multi-indices $\a\in\N^N$, for some constant $C_\a$ depending only on $\a$. 

\begin{theorem}\label{T:2.2}  Let $u(t,x), \hat u(t,x)$  be two $\BAP$-entropy solutions of \eqref{e2.1}-\eqref{e2.2} with initial data $u_0(x), \hat u_0(x)$. Then, a.s., for a.e.\ $t\in[0,T]$ we have  
 \begin{equation}\label{e3.11}
\int_{\R^N}|u(t,x)-\hat u(t,x)|\rho_*(x)\,dx\le C\int_{\R^N} |u_0(x)-\hat u_0(x)|\rho_*(x)\,dx,
\end{equation}
where $C$ depends on $T$ and on the flux function, but not on $\om$.
\end{theorem}

\begin{proof} This follows directly from \eqref{e2.5}, using the properties of  $\rho_*$ and the Lipschitz continuity of $\bff$,  by a standard application of Gr\" onwall's inequality.  Indeed, choosing as test function, by usual approximation, $\varphi(s,x)={\bf 1}_{[0,t)}(s)\rho_*(x)$,  whenever $t$ is a Lebesgue point of the left-hand side of \eqref{e3.11}, where ${\bf 1}_{[0,t)}$ denotes the indicator function of the interval $[0,t)$,
we then obtain
\begin{multline*}
\int_{\R^N}|u(t,x)-\hat u(t,x)|\rho_*(x)\,dx\le \int_{\R^N} |u_0(x)-\hat u_0(x)|\rho_*(x)\,dx\\
+C_1\Lip(\bff)\int_0^t\int_{\R^N}|u(s,x)-\hat u(s,x)|\rho_*(x)\,dx\,ds.
\end{multline*}
\end{proof}

\section{Approximate solutions}\label{S:3}

For simplicity we begin by assuming $\bff\in C_b^3$ and $g_k, \nabla g_k, \nabla^2 g_k \in \AP(\R^N)$. 
The initial datum $u_0$ is deterministic and also for simplicity assume 
$$
u_0, \nabla u_0, \nabla^2 u_0\in \AP(\R^N).
$$

We consider the parabolic approximation for \eqref{e2.1}-\eqref{e2.2}
\begin{equation}\label{e3.1}
du+\div \bff(u)\,dt=\ve\Delta u+\Phi_\ve \,dW(t), \quad x\in\R^N, \ t\in(0,T).
\end{equation}

\begin{equation}\label{e3.2}
u(0,x)=u_{0\ve}(x),\quad x\in\R^N,
\end{equation}
where $u_{0\ve}$ is a trigonometric polynomial approximating $u_0$ in $\AP(\R^N)$, and $\Phi_\ve e_k=g_{k\ve}$ where $g_{k\ve}\equiv 0$ for $k>1/\ve$ and $g_{k\ve}$ is a trigonometric polynomial approximating $g_k$ uniformly in $\AP(\R^N)$ for $k\le 1/\ve$, $k\in\N$. 

Let $J_\ve=\int_0^t \Phi_\ve dW$. For $w=u-J_\ve$ we write \eqref{e3.1} as
\begin{equation}\label{e3.3}
w_t+\div \bff(w+J_\ve)-\ve \Delta w=\ve \Delta J_\ve.
\end{equation} 
 We can obtain a solution to \eqref{e3.1}, \eqref{e3.2}, with 
$w, \nabla w, \nabla^2 w, w_t\in C([0,T];L^\infty(\R^N))$ in a standard
way (see, e.g., \cite{Hr}) beginning  by a fixed point argument for the functional
\begin{equation}\label{e3.4}
\LL(w)(t)=K_\ve(t)*u_{0\ve}-\sum_{i=1}^N\int_0^t\po_{x_i} K_\ve(t-s)*f_i(w+J_\ve)(s)\,ds+ \int_0^tK_\ve(t-s)*\Delta J_\ve ,
\end{equation}
where $K_\ve(t,x)=K(\ve t,x)$, where $K$ is the heat kernel.
Moreover, the fixed point argument may be carried out in the Banach space $C([0,T];\AP(\R^N))$, since it is easy to see that $\LL$ takes $C([0,T];\AP(\R^N))$ into itself. So that
the fixed point $w\in C([0,T];\AP(\R^N))$.

In this way we get a solution $w$ of \eqref{e3.3}-\eqref{e3.2} such that 
$$
\|w_t\|_{L^\infty((0,T)\X\R^N)}+\|\nabla^2 w\|_{L^\infty((0,T)\X\R^N)}\le C(\om,\ve),
$$
Also, if $\rho_*$ is given by \eqref{e3.rho}, we can get
$$
\|w_t\|_{C([0,T];L_{\rho_*}^q(\R^N))}+\|\nabla^2 w\|_{C([0,T];L_{\rho_*}^q(\R^N))}\le C(q,\om,\ve),
$$
or each $q\ge 1$. It follows that 
\begin{equation}\label{e3.5}
\text{$u$ and $\nabla^2 u$ are continuous on $[0,T]\X\R^N$},
\end{equation}
\begin{equation}\label{e3.6}
\|u\|_{L^\infty((0,T)\X\R^N)}+\|\nabla^2 u\|_{L^\infty((0,T)\X\R^N)}\le C(\om,\ve),
\end{equation} 
\begin{equation}\label{e3.7}
\|u\|_{C([0,T]; L_{\rho_*}^q(\R^N))}+\|\nabla u\|_{C([0,T];L_{\rho_*}^q(\R^N))}\le C(\om,\ve),
\end{equation} 
for $q\ge1$, and $u$ satisfies \eqref{e3.1}, \eqref{e3.2}. 

\begin{lemma}
Let $u$ be the solution of \eqref{e3.1}, \eqref{e3.2} constructed above. Then, 
\begin{equation}\label{e3.120}
\bbE\sup_{0\le t\le T}\|u(t)\|_{\B^2}^2\le 2\|u_0\|_{\B^2}^2+CD_0T,
\end{equation}
where $D_0$ is as in \eqref{e1.3} and $C>0$ is a universal constant.

Moreover, we also have that
\begin{equation}\label{e3.8}
\bbE(\sup_{0\le t\le T}\|u\|_{L_{\rho_*}^2(\R^N)}^2)+\ve\bbE\left(\int_0^T\int_{\R^N}|\nabla u|^2\rho_*\,dx\,dt\right)\le C,
\end{equation}
 where $C$ depends on the flux functions and on the $g_k$'s but is independent of $\ve$. In particular,
$u$, $\ve^{1/2} \nabla u$ are bounded in $L^2(\Om\X(0,T); L_\loc^2(\R^N))$ uniformly with respect to $\ve$. 
\end{lemma}
\begin{proof}
Applying the It\^o's formula, for any nonnegative test function $\varphi\in C_c^\infty(\R^N)$ we have that
\begin{equation}\label{e3.120'}
\begin{aligned}
\int_{\mathbb{R}^N}&|u(t,x)|^2\varphi\, dx \\
&=\int_{\mathbb{R}^N}|u_0(x)|^2\varphi(x)\, dx + \int_0^t\int_{\R^N} 2\mathbf{F}(u)\cdot\nabla\varphi\, dx\, ds  \\
& -\int_0^t\int_{\mathbb{R}^N}2\ve |\nabla u|^2\varphi\, dx\, ds -\int_0^t\int_{\mathbb{R}^N}2\varepsilon u\nabla u\cdot \nabla \varphi\, dx\, ds \\
&+\sum_{k\ge 1}\int_0^t\int_{\mathbb{R}^N}2 u\, g_k\, \varphi dx\, d\beta_k(s)+\int_0^t\int_{\mathbb{R}^N} G^2(x)\varphi(x)\, dx\, ds,
\end{aligned}
\end{equation}
where $\mathbf{F}'(u)=u\bff'(u)$ and $G^2$ is as in \eqref{e1.3}, and we omit the superscript $\ve$ for simplicity. We then choose $\varphi$ by setting 
$\varphi(x)=R^{-N}g(x/R)$, with $g\in C_c^\infty(\R^N)$ such that $0\le g(y)\le 1$, $g(y)\equiv1$ in the cube $C_1$, $g(y)=0$ in the complement of the cube $C_k$, $k>1$, as
in the proof of Proposition~\ref{P:2.1}. Making $R\to\infty$ we obtain
\begin{equation}\label{e3.120*}
\|u(t)\|_{\B^2}^2\le \|u_0\|_{\B^2}^2
+\sum_{k\ge 1}\int_0^t 2 \la u, g_k\ra_{\B^2}\, d\beta_k(s)+D_0t.
\end{equation}
Then, we take the $\sup_{0\le t\le T}$ in both sides and apply Bukholder-Davis-Gundy inequality and then Young's inequality to obtain 
\begin{align*}
\bbE\sup_{0\le t\le T}\|u(t)\|_{\B^2}^2&\le  \|u_0\|_{\B^2}^2 +C\bbE\left(\int_0^T  \bigl(\|u\|_{\B^2}^2\sum_{k\ge 1}\| g_k\|_{\B^2}^2\bigr)\, ds\right)^{1/2}+ D_0T\\
   &\le\|u_0\|_{\B^2}^2 +C\bbE\left(\int_0^T  \bigl(\|u\|_{\B^2}^2 D_0\bigr)\, ds\right)^{1/2}+ D_0T\\
   &\le\|u_0\|_{\B^2}^2 +C\bbE\left(\sup_{0\le t\le T}\|u\|_{\B^2}^2D_0T\right)^{1/2}+D_0T\\
   &\le\|u_0\|_{\B^2}^2 +\frac{1}{2}\bbE\sup_{0\le t\le T}\|u(t)\|_{\B^2}^2+ CD_0T,
\end{align*}
which yields \eqref{e3.120}.

Similarly, starting from \eqref{e3.120'} we see that we may take a sequence of $\varphi$'s converging to $\rho_*$, so that, using \eqref{e2.200} and noting that $2\ve|u\nabla u|\le \ve|\nabla u|^2+|u|^2$, we obtain 
\begin{multline}\label{e3.120*'}
\|u(t)\|_{L_{\rho_*}^2(\R^N)}^2+\ve\int_0^t\int_{\R^N}|\nabla u|^2\rho_*\,dx\,dt \le \|u_0\|_{L_{\rho_*}^2(\R^N)}^2 \\
+C\int_0^t\|u(s)\|_{L_{\rho_*}^2(\R^N)}^2\,ds +\sum_{k\ge 1}\int_0^t 2 \la u, g_k\ra_{L_{\rho*}^2}\, d\beta_k(s)+D_0t,
\end{multline}
where we used the fact that $|\mathbf{F}(u)|\le \frac{\Lip(f)}{2} |u|^2$.
Taking expectation and applying Gronwall's inequality we get 
$$
\sup_{0\le t\le T}\bbE\|u(t)\|_{L_{\rho_*}^2(\R^N)}^2\le \|u_0\|_{L_{\rho*}^2(\R^N)}^2+D_0T.
$$

Then coming back to \eqref{e3.120*'}, taking $\sup_{0\le t\le T} $ in both sides, again applying   Bukholder-Davis-Gundy inequality and then Young's inequality we finally obtain \eqref{e3.8}.
\end{proof}

Let us now consider the kinetic formulation for \eqref{e3.1}. Namely, if $\chi_u(\xi)=1_{\xi< u}(\xi)-1_{\xi<0}$, for $S\in C ^2(\R)$, we have
$$
S(u)-S(0)=\int_{\R}S'(\xi)\chi_u(\xi)\,d\xi.
$$
Using this fact and It\^o's formula applied to $S(u)$ we deduce that $\ff(t,x,\xi)=\chi_{u(t,x)}(\xi)$ satisfies
\begin{equation}\label{e3.9}
\ff_t+\abf(\xi)\cdot\nabla_x\ff -\ve\Delta \ff= (\ve|\nabla_x u|^2\d_{\xi=u}(\xi)- \frac12 G_\ve^2\d_{\xi=u}(\xi))_\xi - 
\d_{\xi=u} \Phi\,dW,
\end{equation}
where $\abf(\xi)=\mathbf{f}'(\xi)$.

In order to make use  of a regularity estimate from the stochastic averaging lemma in \cite{GH}, we need to impose a non-degeneracy condition ({\em cf.} \cite{GH}).
For $\k=(\k_1,\cdots,\k_N)\in\Z^N$, let $|\k|^2=\k_1^2+\k_2^2+\cdots+\k_N^2$.  For $J\gtrsim1$, $\d>0$, 
we suppose there exist $\a\in(0,1)$ such that
\begin{equation}\label{e1.GH}
\iota_*(\d,J):=\sup_{\tiny{\begin{matrix} \tau\in\R, \k\in\Z^N\\ |\k|\sim J\end{matrix}}}|\{\xi\in\R\,:\, |\tau+\k\cdot \abf(\xi)|\le \d\}|
\lesssim \left(\frac{\d}{J}\right)^\a,
\end{equation}
 where, for a measurable subset $A\subset \R$,  we denote $|A|$ the one-dimensional Lebesgue  measure of $A$.
Here we employ the usual notation $x\lesssim y$, if $x\le Cy$, for some absolute constant $C>0$, and $x\sim y$, if $x\lesssim y$ and $y\lesssim x$.

Assuming \eqref{e1.GH} holds, we can then obtain a regularity estimate from the stochastic averaging lemma in \cite{GH} of the type
\begin{equation}\label{e3.10}
\bbE\|\phi u\|_{L^1(0,T; W^{s,r}(\R^N))}\le C( \|\phi u_0\|^3_{L_{\om,x}^3}+1),
\end{equation}
for each $\phi\in C_c^\infty(\R^N)$.  In particular, given an open bounded set, with smooth boundary, $\mathcal{O}$, there exists a constant $C_{\mathcal{O}}>0$, independent of $\ve>0$, such that
 \begin{equation}\label{e3.10'}
\bbE\| u^\ve\|_{L^1(\Om\X[0,T]; W^{s,r}(\mathcal{O}))}\le C_{\mathcal{O}},
\end{equation}

As to the existence, we apply a reasoning similar to the one in \cite{FLMNZ}, which follows the method in \cite{DHV} (see also \cite{Ha}). Namely: (i) to apply Kolmogorov's continuity lemma; (ii)  to prove of the tightness of the laws which, by Prokhorov's theorem, implies the compactness of the laws in the weak topology of measures; (iii) to apply Skorokhod's representation theorem; (iv) to show that the limit a.e.\  given by Skorokhod's representation theorem is a martingale entropy solution;  (iv) to apply the Gy\" ongy-Krylov criterion for convergence in probability, using the uniqueness of the solution of \eqref{e2.1}-\eqref{e2.1}, therefore  obtaining the convergence in $L_\loc^1$ of the solutions of \eqref{e3.1}-\eqref{e3.2} to a $L_\loc^1$ function which is an entropy solution of \eqref{e2.1}-\eqref{e2.2}.  We now line up the main results that follow the just described streamline. To begin with, through the application of Kolmogorov's continuity lemma, we have the following.

\begin{proposition}\label{P:3.2} Let $u^\ve$ be the solution of \eqref{e3.1}-\eqref{e3.2}. For all $\l\in(0,1/2)$, for each 
$\phi\in C_c^\infty(\R^N)$, there exists a constant $C_\phi>0$, depending on $\phi$ but  independent of $\ve>0$, such that 
 for all $\ve\in(0,1)$
 $$
 \bbE\|\phi u^\ve\|_{C^\l([0,T]; H^{-1}(\R^N))}\le C_\phi.
 $$
 \end{proposition}

In particular, for a fixed $1<r<2$,  for all $\ve\in(0,1)$, for all open bounded with smooth boundary 
$\mathcal{O}$ ,we have 
\begin{equation}\label{e3.100}
 \bbE\|u^\ve\|_{C^\l([0,T]; W^{-1,r}(\mathcal{O}))}\le C_{\mathcal{O}},
 \end{equation}
 for some constant $C_{\mathcal{O}}>0$ independent os $\ve>0$.
 
For a fixed $1<r<2$ such that the regularity estimate \eqref{e3.10}, from \cite{GH}, holds,  let us denote, 
$$
\mathcal{X}=L^1(0,T; L_\loc^r(\R^N))\cap C([0,T]; W_\loc^{-1,r}(\R^N)).
$$
First, we recall that $u\in \mathcal{X}$ if, for each $\phi\in C_c^\infty(\R^N)$, $\phi u\in L^1(0,T; L^r(\R^N))\cap C([0,T];W^{-1,r}(\R^N))$.  Also, convergence of a sequence $u_n\to u$ in 
$\mathcal{X}$ means that for all $\phi\in C_c^\infty(\R^N)$, $\phi u_n\to \phi u$ in  
$L^1(0,T; L^r(\R^N))\cap C([0,T];W^{-1,r}(\R^N))$. Observe that we can endow $\mathcal{X}$ with a metric with respect to which it becomes a metric separable space. Indeed, for $\nu\in\N$, let 
$\mathcal{O}_\nu$ be the open ball  of radius $\nu$ around the origin in $\R^N$, and let $\phi_\nu\in C_c^\infty(\R^N)$, $0\le \phi_\nu\le 1$, with $\phi_\nu\equiv1$ on $\mathcal{O}_\nu$ and $\phi_\nu\equiv 0$, outside $\mathcal{O}_{\nu+1}$. 
 For $u\in\mathcal{X}$, let $\rho_\nu(u)$ be the norm of $\phi_\nu u$ in $L^1(0,T; L^r(\R^N))\cap C([0,T];W^{-1,r}(\R^N))$. We can then define the following metric in $\mathcal{X}$,
$$
\d(u,v)=\sum_{\nu=1}^\infty 2^{-\nu} \frac{\rho_\nu(u-v)}{1+\rho_\nu(u-v)}, \quad u,v \in\mathcal{X}.
$$

 Concerning the tightness of the laws $\mu_{u^\ve}$, $\ve\in(0,1)$,  associated to the solutions of \eqref{e3.1}-\eqref{e3.2}, $0<\ve<1$, we have the following result ({\em cf.}, e.g., the proof of proposition~5.3 in \cite{FLMNZ}). 
   
\begin{proposition}\label{P:3.2''} Let  $\mu_{u^\ve}$ be the law defined in $\mathcal{X}$ associated with $u^\ve$. The set 
$\{\mu_{u^\ve}\,:\, \ve\in(0,1)\}$ is tight and, therefore, relatively weakly compact in $\mathcal{X}$.
\end{proposition}

\begin{proof} Let us define
\begin{multline*}
K_R=\{ u\in\mathcal{X}\,:\, 
 \|u\|_{L^\infty(0,T; L_{\rho_*}^2(\R^N))}\le R,\\ 
 \|\phi_\nu u\|_{C^\l([0,T];W^{-1,r}(\R^N))}\le 2^\nu(C_{\mathcal{O}_{\nu+1}} +1)R, 
 \\  \|\phi_\nu u\|_{L^1(0,T;W^{s,r}(\R^N))} \leq 2^\nu (C_{\mathcal{O}_{\nu+1}} + 1) R,
 \quad \forall \nu \geq 1\},
 \end{multline*}
where $C_{\mathcal{O}_{\nu+1}}$ is greater than or equal to the constant $C_{\mathcal{O}}$ in
\eqref{e3.10'} and \eqref{e3.100} for $\mathcal{O}=\mathcal{O}_{\nu+1}$, and $\phi_\nu$ is as above. We claim that $K_R$ is a relatively compact subset of $\mathcal{X}$. Indeed, if $\psi^k$ is a sequence in $K_R$, then $\|\psi^k\|_{L^\infty(0,T; L_{\rho_*}^2(\R^N))}\le R$, and
for all $\nu\in\N$, we have that $\phi_\nu \psi^k$ is bounded in 
$C^\l([0,T];W^{-1,r}(\mathcal{O}_{\nu+1}))\cap L^1(0,T;W^{s,r}(\mathcal{O}_{\nu+1}))$. 
Since $\|\psi^k\|_{L^\infty(0,T; L_{\rho_*}^2(\R^N))}\le R$, we can extract a subsequence, still denoted $\psi^k$,
and a $\psi\in L^\infty(0,T; L_{\rho_*}^2(\R^N))$ such that $\psi^k\wto \psi$ in the weak* topology of $L^\infty(0,T; L_{\rho_*}^2(\R^N))$,
notice that this space is the topological dual of $L^1(0,T; L_{\rho_*}^2(\R^N))$.
For all $\nu\in\N$, for a.e.\ $t\in[0,T]$, we have that $\|\phi_\nu \psi^k(t)\|_{L^2(\R^N)}\le C_{\phi_\nu}R$, for all $k\in\N$. In particular, we can find a dense set in $[0,T]$ and a subsequence, still denoted $\psi^k$, such that $\phi_\nu\psi_k(t)\to \phi_\nu\psi(t)$ in $W^{-1,2}(\R^N)$ strongly , for a dense set of $t\in[0,T]$, for all 
$\nu\in\N$, and so, also in  $W^{-1,r}(\R^N)$.  Since $\phi_\nu\psi^k$ is bounded in $C^\l([0,T];W^{-1,r}(\mathcal{O}_{\nu+1}))$, we deduce that  $\phi_\nu\psi^k\to\phi_\nu \psi$ strongly in $C([0,T];W^{-1,r}(\R^N))$, for all $\nu\in\N$.  
 On the other hand,  by interpolation we have, for all $\varphi\in C_c^\infty(\mathcal{O}_{\nu+1})$,  
$$
\|\varphi\|_{L^1(0,T;W^{2,r}(\cO_{\nu+1}))}\le \|\varphi\|_{L^1(0,T;W^{1,r}(\cO_{\nu+1}))}^{s/(1+s)}\|\varphi\|_{L^1(0,T;W^{2+s,r}(\cO_{\nu+1}))}^{1/(1+s)}.
$$
Then, by density, taking $\varphi=(-\Delta)^{-1}(\phi_\nu\psi_k)$, where by $-\Delta$ we mean the minus Laplacian operator with 0 Dirichlet condition on $\po\cO_{\nu+1}$, we conclude that $\phi_\nu\psi_k$ strongly converges in $L^1(0,T;L^r(\cO_{\nu+1}))$, using that $(-\Delta)^{-1}$ isomorphically takes  $L^1(0,T;L^r(\cO_{\nu+1}))$ onto
$L^1(0,T; W^{2,r}\cap W_0^{1,r}(\cO_{\nu+1}))$.

 In this way, by a standard diagonal argument, we obtain a subsequence of $\psi^k$, still denoted $\psi^k$, such that $\phi_\nu\psi^k$ converges in
$C([0,T];W^{-1,r}(\R^N))\cap L^1([0,T];L^r(\R^N))$, for all $\nu\in\N$, which implies the compactness of $K_R$ in $\mathcal{X}$. 

As for the tightness of $\mu_{u^\ve}$, we have
\begin{multline*}
\mu_{u^\ve}(\mathcal{X}\setminus K_R)\le
\bbP\Big(\|u^\ve\|_{L^\infty(0,T;L_{\rho_*}^2(\R^N))}>R\Big)\\
+\sum_{\nu=1}^\infty \bbP\Big(\|\phi_\nu u^\ve\|_{C^\l([0,T];W^{-1,r}(\R^N))}>2^\nu(C_{\mathcal{O}_{\nu+1}} +1)R\Big)\\
+\sum_{\nu=1}^\infty \bbP\Big(\|\phi_\nu u^\ve\|_{L^1(0,T;W^{s,r}(\R^N))}>2^\nu(C_{\mathcal{O}_{\nu+1}} +1)R\Big)\\
\le \frac1{R^2}\bbE\sup_{[0,T]}\|u^\ve(t)\|_{L_{\rho_*}^2(\R^N)}^2\\
+\sum_{\nu=1}^\infty \frac1{2^\nu(C_{\mathcal{O}_{\nu+1}} +1)R}\bbE\left(\|\phi_\nu u^\ve\|_{C^\l([0,T];W^{-2,r}(\R^N))}+\|\phi_\nu u^\ve\|_{L^1(0,T;W^{s,r}(\R^N))}\right)\\
\le \frac{C}{R^2}+\frac2{R},
\end{multline*}
by using \eqref{e3.8},   \eqref{e3.10'} and \eqref{e3.100}, which implies the tightness of $\mu_{u^\ve}$.  

\end{proof}

With Proposition~\ref{P:3.2''} at hand, we apply Prokhorov's theorem to obtain a subsequence $u^n$ such that $\mu_{u^n}$ weakly converges over $\mathcal{X}$. We can then apply Skorokhod's theorem and obtain a further subsequence still denoted $u^n$, a new probability space $(\tilde \Om,\tilde\bbP)$, and a subsequence $\tilde u^n$, with $\mu_{\tilde u^n}=\mu_{u^n}$, such that $\tilde u^n:\tilde \Om\to\mathcal{X}$ converges a.s.\ to $\tilde u:\tilde \Om\to \mathcal{X}$.  

\begin{proposition}\label{P:3.2'''} There exists a probability space $(\tilde \Om,\tilde \F,\tilde \bbP)$ with a sequence of $\mathcal{X}$-valued random variables $\tilde u^n$, $n\in\N$, and
$\tilde u$ such that:
\begin{enumerate}
\item[(i)] the laws of $\tilde u^n$ and $\tilde u$ under $\tilde \bbP$ coincide with $\mu^n$ and $\mu$, respectively;
\item[(ii)] $\tilde u^n$ converges $\tilde \bbP$-almost surely to $\tilde u$ in the topology of $\mathcal{X}$.

\end{enumerate}
\end{proposition}

We then define yet another probability space $\overline{\Om}=\Om\X\tilde \Om$ with the product probability measure $\overline{\bbP}=\bbP\X\tilde{\bbP}$, the $\s$-algebra $\overline{\F}$ as the product  $\s$-algebra generated by $\tilde \F\X \F$,    and, from the Wiener process $W(t)$ in $\Om$, we define the Wiener process $\overline{W}(t)$ in $\overline{\Om}$ trivially by $\overline{W}(t)(\om,\tilde \om)=W(t)(\om)$, for $(\om,\tilde \om)\in\overline{\Om}=\Om\X\tilde\Om$; clearly, $\overline{W}$ has the same law as $W$.  Defining $\bar u^n:\overline{\Om}\to\mathcal{X}$ by $\bar u^n(\om,\tilde \om)=\tilde u^n(\tilde \om)$, we have that $\mu_{\bar u^n}=\mu_{\tilde u^n}=\mu_{u^n}$. Also, $\bar u^n$ converges a.s.\ in $\overline{\Om}$ to the random variable $\bar u:\overline{\Om}\to \mathcal{X}$ defined by $\bar u(\om,\tilde \om)=\tilde u(\tilde\om)$. We define a  filtration $\overline{\F}_t$ for $(\overline{\Om}, \overline{\F},\overline{\bbP})$ in the following way ({\em cf.} \cite{Ha}).
For each $t\in[0,T]$, the restriction map $\rho_t: C([0,T]; W_\loc^{-1,r} (\R^N)\X C([0,T];\mathfrak{U}_0)\to C([0,t]; W_\loc^{-1,r} (\R^N))\X C([0,t];\mathfrak{U}_0)$, $(v,W)\mapsto (v,W)|[0,t]$, is a continuous map. Here, $\mathfrak{U}_0$ is the Hilbert space where the cylindrical Wiener process $W(t)$ is well defined. So, we define as $\overline{\F}_t=\s(\rho_t\bar u,\rho_t\overline{W})$, the $\s$-algebra of subsets of $\overline{\Om}$ generated 
by the function $(\rho_t\bar u, \rho_t\overline{W}):\overline{\Om}\to C([0,t];W_\loc^{-1,r}(\R^N))\X C([0,t];\mathfrak{U}_0)$, and we denote also by $(\overline{\F}_t)_{t\ge0}$ the corresponding augmented filtration, i.e., the smallest complete right-continuous filtration containing $(\overline{\F_t})_{t\ge0}$.

\begin{definition}\label{D:3.1} We say that $\tilde u$ is an entropy martingale solution of \eqref{e2.1}-\eqref{e2.2} if, for some probability space equipped with a filtration $(\overline{\Om}, \overline{\F}, (\overline{\F}_t), \overline{\bbP})$ and some cylindrical Wiener process $\overline{W}(t)=\sum_{i=1}^\infty \bar \b_k e_k$,
with respect to the filtration $(\overline{\F}_t)$,
with $\{e_k\}_{k\ge1}$  a complete orthonormal system in a Hilbert space $H$  and $\bar \b_k$, $k\in\N$, independent Brownian motions in $(\overline{\Om}, \overline{\F}, (\overline{\F}_t), \overline{\bbP})$,  
if $\bar w=\bar u-\bar J$, with $\bar J:=\sum_{k=1}^\infty g_k\bar\b_k$, satisfies Definition~\ref{D:2.1} with $u$, $J$, $W$, replaced by $\bar u$, 
$\bar J$,  $\overline{W}$.
 \end{definition}

\begin{proposition}\label{P:3.new}  The limit $\bar u$ with $(\overline{\Om}, \overline{\F}, (\overline{\F}_t)_{t\ge0}, \overline{\bbP},\overline{W})$ is  an entropy martingale solution of \eqref{e2.1}-\eqref{e2.2}  in the sense of Definition~\ref{D:3.1}.
\end{proposition}

 
\begin{proof} We observe that $\bar u^n$ and  the limit $\bar u$ also satisfy \eqref{e3.120} and \eqref{e3.8}. Hence, the limit $\bar u$ is a.s.\ in $L^\infty((0,T);\B^1(\R^N)\cap L_\loc^1(\R^N))$.  We then have that $\bar u$ is an $L_\loc^1\cap \B^1$-valued adapted process. 

We now prove item (3) of Definition~\ref{D:2.1}. For any continuous $\g:\mathcal{X}\X\mathcal{W}\to [0,1]$, with $\mathcal{W}=C([0.1];\mathfrak{U}_0)$, and any $0\le \varphi\in C_c^\infty([0,\infty)\X\R^N)$,
\begin{multline*}
\bbE\g(\bar u, \overline{W})\int_0^T\int_{\R^N} \{(|\bar w-\a|-|u_0-\a|)\varphi_t\\
+\sgn(\bar w-\a)(\bff(\bar w+\bar J)-\bff(\a+\bar J))\cdot\nabla\varphi
 -\sgn(\bar w-\a)\bff'(\a+\bar J)\nabla \bar J\varphi\}\,dx\,dt\\
=\lim_{n\to\infty} \bbE\g({\bar u}^n, \overline{W})\int_0^T\int_{\R^N} \{(|{\bar w}^n-\a|-|u_0-\a|)\varphi_t\\
+\sgn({\bar w}^n-\a)(\bff({\bar w}^n+\bar J_n)-\bff(\a+\bar J_n))\cdot\nabla\varphi
 -\sgn({\bar w}^n-\a)\bff'(\a+\bar J_n)\nabla \bar J_n\varphi\}\,dx\,dt\\
 =
\lim_{n\to\infty} \bbE\g( u^n, W)\int_0^T\int_{\R^N} \{(| w^n-\a|-|u_0-\a|)\varphi_t\\
+\sgn(w^n-\a)(\bff(w^n+ J_n)-\bff(\a+ J_n))\cdot\nabla\varphi 
-\sgn(w^n-\a)\bff'(\a+ J_n)\nabla J_n\varphi\}\,dx\,dt\\
\ge -\lim_{n\to\infty}\ve_n\bbE\g( u^n, W)\int_0^T\int_{\R^N}\{|w^n-\a|\Delta\varphi+\sgn(w^n-\a)\Delta J_n\varphi \}\,dx\,dt\\
\ge -C(\|\Delta\varphi\|_\infty,\|\varphi\|_\infty, \diam(\supp\varphi), D_0,\a)
 \\
\X\lim_{n\to\infty} \ve_n\left(\bbE\sup_{0\le t\le T}\|u_n\|_{L_{\rho_*}^2}^2+\sum_{k=1}^\infty \a_k\bbE|\b_k(T)|+1\right)
\to0\quad \text{as $n\to\infty$}.
 \end{multline*}
Since $\g:\mathcal{X}\X\mathcal{W}\to [0,1]$ is arbitrary, we conclude that, a.s., 
\begin{multline}\label{e3.121}
\int_0^T\int_{\R^N} \{(|\bar w-\a|-|u_0-\a|)\varphi_t+\sgn(\bar w-\a)(\bff(\bar w+\bar J)-\bff(\a+\bar J))\cdot\nabla\varphi\\
 -\sgn(\bar w-\a)\bff'(\a+\bar J)\nabla \bar J\varphi\}\,dx\,dt\ge0,
\end{multline} 
for all $0\le \varphi\in C_c^\infty([0,\infty)\X\R^N)$, which verifies item (3) of Definition~\ref{D:2.1}.

{}From \eqref{e3.121}, with $\a=0$, let us take a test function as in the proof of Proposition~\ref{P:2.1},  
$\varphi(t,x)=
R^{-N}\chi_\nu(t)g(x/R)$, where $\chi_\nu(t)$ is a smooth approximation of ${\bf1}_{[0,t_0)}$, and 
$t_0$ is a Lebesgue point of
$$
\frac1{R^N}\int_{\R^N}|\bar w(t,x)| g(x/R)\,dx.
$$ 
 After making $\nu\to0$ we get for a.a.\ $\bar \om\in\overline{\Om}$ 
 \begin{multline*} 
 \frac1{R^N}\int_{\R^N} (|\bar w(t_0)|-|u_0|)\,dx+\frac1{R^{N+1}}\int_0^{t_0}\int_{\R^N} \sgn(w)(\bff(\bar    w(t)+\bar J(t))-\bff(\bar J(t)))\cdot \nabla g(x/R)\,dx\\
 \le \frac1{R^N}\int_0^{t_0}\int_{\R^N} |\bff'(\bar J)\nabla \bar J| g(x/R)\,dx\,dt,
 \end{multline*}
 which implies, by making $R\to\infty$, 
 \begin{equation}\label{e3.new}
 \Medint_{\R^N}|\bar w(t_0,x)|\,dx\le \Medint_{\R^N}|u_0(x)|\,dx
 +\int_0^{t_0}\Medint_{\R^N}|\bff'(J)||\nabla J|\,dx\,dt,
 \end{equation}
 for a.e.\ $t_0\in[0,T]$.
 In particular, \eqref{e3.new} proves that $\bar w\in L^\infty([0,T]; \B^1(\R^N))$, which we had already obtained from \eqref{e3.120}, but \eqref{e3.new} gives an explicit bound.  Similarly,
 from \eqref{e3.121}, with $\a=0$, using  $\varphi(t,x)=\rho_*(x)\chi_\nu(t)$, as test function and Gr\"onwall's inequality, we obtain, for a.e.\ $t_0\in[0,T]$, 
 \begin{equation}\label{e3.new2}
 \int_{\R^N}|\bar w(t_0,x)|\rho_*(x)\,dx\le C(T)\int_{\R^N}|u_0(x)|\rho_*(x)\,dx,
 \end{equation}
 which proves that  $\bar w\in L^\infty([0,T]; L_\loc^1(\R^N))$ and, so, item (2) of Definition~\ref{D:2.1} is verified.

 In \eqref{e3.121}, we choose $\varphi(t,x)=\phi(t,x)\frac1{R^N}g(x/R)$, with $g$ as above, $0\le \phi\in C_c^\infty([0,\infty); \AP^\infty(\R^N))$, where $\AP^k(\R^N)$, $k\in\N\cup\{0,\infty\}$, is the space of functions in $C_b^k(\R^N)$ such that all derivatives up to the order $k$ belong to $\AP(\R^N)$. Then, making $R\to\infty$, we get
 \begin{multline}\label{e3.122}
\int_0^T\int_{\bbG_N} \{(|\bar w-\a|-|u_0-\a|)\phi_t+\sgn(\bar w-\a)(\bff(\bar w+\bar J)-\bff(\a+\bar J))\cdot\nabla\phi\\
 -\sgn(\bar w-\a)\bff'(\a+\bar J)\nabla \bar J\phi\}\,d\mm(y)\,dt\ge0,
\end{multline}  
 
 Now, since the $(N+1)$-dimensional Lebesgue measure of the subset of $[0,T]\X\R^N$  where $|\bar w|> |\a|$ goes to zero when $|\a|\to\infty$, because $\bar w\in L^\infty([0,T]; L_{\rho_*}^2(\R^N))$, we get from \eqref{e3.121}, letting $|\a|\to\infty$,     
\begin{equation}\label{e3.123}
\int_0^T\int_{\R^N} \{(\bar w-u_0)\varphi_t+\bff(\bar w+\bar J)\cdot\nabla\varphi\}\,dx\,dt=0,
\end{equation} 
for all $\varphi\in C_c^\infty([0,\infty)\X\R^N)$.

Similarly, from \eqref{e3.122} we get 
 \begin{equation}\label{e3.124}
\int_0^T\int_{\bbG_N} \{(\bar w-u_0)\phi_t+\bff(\bar w+\bar J)\cdot\nabla\phi\}\,d\mm(y)\,dt=0,
\end{equation} 
for all  $\phi\in C_c^\infty([0,\infty); \AP^\infty(\R^N))$. 

{}From \eqref{e3.123} and \eqref{e3.124}, by choosing suitable test functions, we prove in a standard way
that $t\mapsto \bar w(t)$ is weakly continuous in $L_\loc^1(\R^N)\cap \B^1(\R^N)$, recalling that 
$\B^1(\R^N)$ can be identified with $L^1(\bbG_N)$, which then shows that $\bar u$ also satisfies item~(1) of Definition~\ref{D:2.1}.  Hence, $\bar u$ is indeed a martingale entropy solution of \eqref{e2.1}-\eqref{e2.2}.
\end{proof}

\medskip
Now, since we have the uniqueness of such solution by Theorem~\ref{T:2.2}, we can apply Gy\"ongy-Krylov's  criterion  for convergence in probability to conclude that the whole original sequence $u^\ve$ converges almost surely 
in $L_\loc^1((0,T)\X\R^N)$ to an entropy solution of \eqref{e2.1}-\eqref{e2.2} in the sense of Definition~\ref{D:2.1}. In this way, we have proved the existence of a  $\BAP$-entropy solution to \eqref{e2.1}-\eqref{e2.2} in the sense of Definition~\ref{D:2.1}, for initial data in $\AP(\R^N)$, which can easily be generalized to any initial data in $L^\infty\cap \B^1(\R^N)$ due to the stability results of Proposition~\ref{P:2.1} and Theorem~\ref{T:2.2}.   

\section{$L^1(\bbG_N)$-entropy solutions and $L^1(\bbG_N)$-semigroup solutions}\label{S:4}

In order to study invariant measures on $\B^1(\R^N)$ we need to formulate the definition of solution to \eqref{e1.1}-\eqref{e1.2} in $\bbG_N$ the Bohr compact  associated with $\AP(\R^N)$ such that $\AP(\R^N)\sim C(\bbG_N)$. The latter implies the isometric isomorphism 
$\B^1(\R^N)\sim L^1(\bbG_N)$. We denote by $C^k(\bbG_N)$ the space corresponding to $\AP^k(\R^N)$
by the same isomorphism, $k\in\N\cup\{0,\infty\}$.
 
 \begin{definition}\label{D:4.1} Let $T>0$ be given. A $L^1(\bbG_N)$-valued stochastic process adapted
to $\{\F_t\}$ is said to be a $L^1(\bbG_N)$-entropy solution of \eqref{e1.1}-\eqref{e1.2} if, for almost all $\om\in\Om$,
\begin{enumerate}
 \item $u\in L^1(\bbG_N)$-weakly continuous on $[0,T]$,
 \item $u\in L^\infty([0,T]; L^1(\bbG_N))$,
 \item for all nonnegative $\tilde\varphi\in C_0^\infty ([0,T)\X\bbG_N)$ and $\a\in\R$
 \begin{multline}\label{e4.3}
 \int_0^T\int_{\bbG_N} \{(|w-\a|-|u_0-\a|)\tilde\varphi_t+\sgn(w-\a)(\bff(w+J)-\bff(\a+J))\cdot\nabla\tilde\varphi\\
 -\sgn(w-\a)\bff'(\a+J)\nabla J\tilde\varphi\}\,d\mm(z)\,dt\ge0,
 \end{multline}
 \end{enumerate}
 \end{definition}

\begin{theorem}\label{T:3new} Assume \eqref{e1.GH} holds. Given, $u_0\in L^1(\bbG_N)$, there exists a $L^1(\bbG_N)$-entropy solution of  \eqref{e1.1}-\eqref{e1.2} in the sense of Definition~\ref{D:4.1} which is the limit in   $L^1(\Om;  L^\infty((0,T) ;L^1(\bbG_N)))$  of
$\BAP$-entropy solutions of \eqref{e2.1}-\eqref{e2.2} with initial data in $\AP(\R^N)$. 
\end{theorem}

\begin{proof} Indeed, existence of a $L^1(\bbG_N)$-entropy solution of \eqref{e1.1}-\eqref{e1.2}  in the sense of Definition~\ref{D:4.1},
when $u_0\in\AP(\R^N)$,  follows from the existence of $\BAP$-entropy solution of \eqref{e2.1}-\eqref{e2.2} in the sense of Definition~\ref{D:2.1}, proved above,  by suitably choosing in \eqref{e2.3} a test function of the form
$$
\varphi(t,x)= R^{-N}g(x/R) \tilde\varphi(t,x),
$$ 
where $g$ is as in the proof of Proposition~\ref{P:2.1}, $\varphi\in C_c^\infty([0,T); \AP^\infty(\R^N))$. In this way we obtain that the $\BAP$-entropy solution satisfies the item (3)
of Definition~\ref{D:4.1}, that is, it verifies  \eqref{e4.3} for all nonnegative $\tilde\varphi\in C_0^\infty ([0,T)\X\bbG_N)$ and $\a\in\R$. 

We next observe that making $\a\to+\infty$ in \eqref{e4.3}, splitting the integral over $\bbG_N$ in two parts, one over $\{w\le \a\}$ and other over $\{w>\a\}$, and, after, also making $\a\to-\infty$ and proceeding similarly, we obtain the following integral equation
\begin{equation}\label{e4.3eq}
 \int_0^T\int_{\bbG_N} \{w\varphi_t+\bff(w+J)\cdot\nabla\varphi\}\,d\mm(z)\,dt
 =\int_{\bbG_N} u_0(z) \varphi(0,z )\,d\mm(z),
 \end{equation}
for all $\varphi\in C_0^\infty ([0,T)\X\bbG_N)$. The fact that \eqref{e4.3eq} holds for all $\varphi\in C_0^\infty ([0,T)\X\bbG_N)$ implies, in turn, in a standard way, the item (1) of Definition~\ref{D:4.1}. 

Finally, we may verify the item~(2) of Definition~\ref{D:4.1} also as a consequence of item (3) of Definition~\ref{D:4.1} by
obtaining an inequality as \eqref{e3.new}, in the same way as latter was obtained from \eqref{e3.121}. This concludes the proof that $\BAP$-entropy solutions are also (or extend to) $L^1(\bbG_N)$-entropy solutions. 

{}Also, given two  $\BAP$-entropy
solutions $u(t,x), v(t,x)$, with initial data $u_0, v_0\in \AP(\R^N)$, we get from  \eqref{e2.4}
\begin{equation}\label{e4.10} 
\int_{\bbG_N} |u(t,z)-v(t,z)|\,d\mm(z)\le \int_{\bbG_N}|u_0(z)-v_0(z)|\,d\mm(z).
\end{equation}
Now, from \eqref{e4.10} we can extend the existence of $L^1(\bbG_N)$-entropy solutions for initial data $u_0\in L^1(\bbG_N)$. Indeed,  if we approximate the initial data $u_0\in L^1(\bbG_N)\sim \B^1(\R^N)$ in $\B^1(\R^N)$ by a sequence $u_{0n}\in \AP(\R^N)$, from \eqref{e4.10} we deduce that the corresponding $\BAP$-entropy solutions $u_n$, in the sense of Definition~\ref{D:2.1}, form a Cauchy sequence in $L^\infty((0,T); L^1(\bbG_N))$, and so, there is $u\in L^\infty((0,T);L^1(\bbG_N))$ such that $u_n\to u$ in $L^\infty((0,T);L^1(\bbG_N))$. This is true a.s.\ in $\Om$ and since by \eqref{e3.new} the norm of the $u_n$'s in $L^\infty((0,T);L^1(\bbG_N))$  is
 bounded by a function in $L^1(\Om)$, we conclude by dominated convergence that 
$$
u_n\to u \quad \text {in $L^1(\Om;L^\infty((0,T); L^1(\bbG_N)))$}.
$$ 
It is then easy to check that the limit $u$ is indeed a $L^1(\bbG_N)$-entropy solution in the sense of Definition~\ref{D:4.1}. Moreover, the contraction property \eqref{e4.10}  extends to any pair of such  $L^1(\bbG_N)$-entropy solutions with initial data in $L^1(\bbG_N)$, obtained as limit in $L^1(\Om; L^\infty([0,T]; L^1(\bbG_N)))$ of  $\BAP$-entropy solutions. In this way we have proved the existence of a $L^1(\bbG_N)$-entropy solution to \eqref{e1.1}-\eqref{e1.2} for any initial data in $L^1(\bbG_N)$. 
\end{proof}

 \begin{definition}\label{D:4.2} Let $T>0$ be given. A $L^1(\bbG_N)$-entropy solution of \eqref{e1.1}-\eqref{e1.2} is said to be a $L^1(\bbG_N)$-semigroup solution if it is the limit in $L^1(\Om;L^\infty((0,T);L^1(\bbG_N)))$ of a sequence of $\BAP$-entropy solutions of \eqref{e2.1}-\eqref{e2.2}, with initial functions converging to $u_0$ in $L^1(\bbG_N)$. 
 \end{definition}
 
 Existence of a $L^1(\bbG_N)$-semigroup solution is guaranteed by Theorem~\ref{T:3new}. The following proposition justifies the introduction of the notion of $L^1(\bbG_N)$-semigroup solution of \eqref{e1.1}-\eqref{e1.2}. The proof is straightforward. 
 
 \begin{proposition}\label{P:4.1} Let $u$ and $v$ be two $L^1(\bbG_N)$-semigroup solutions of \eqref{e1.1}-\eqref{e1.2} with initial functions $u_0,v_0\in L^1(\bbG_N)$. Then \eqref{e4.10} holds.
 \end{proposition}

\section{Reduction to the periodic case}\label{S:5}

In this section and the next one we consider solutions of \eqref{e1.1}-\eqref{e1.2} taking values in a separable subspace of $L^1(\bbG_N)$. More specifically, we consider the  closed real algebra generated by 1 and the complex trigonometrical functions $e^{\pm i 2\pi\l_j\cdot x}$, $j=1,\cdots, P$, with $\{\l_j\,:\, j=1,\cdots,P\}$  a finite set of vectors in $\R^N$ linearly independent over $\Z$. This closed algebra is the closed subspace of $\AP(\R^N)$ formed by the functions of the form $g(y(x))$, with $g\in C(\bbT^P)$, where $\bbT^P$ is the $P$-dimensional torus and $y(x):=(\l_1\cdot x, \cdots, \l_P\cdot x)$. Indeed, it is the completion in the $\sup$-norm of the real trigonometric polynomials of the form 
$$
s(y(x))=\sum_{\bar k\in J}a_{\bar k}e^{i2\pi \bar k\cdot y(x)},
$$
where $J\subset\Z^P$ is a finite set. Since we are considering real trigonometric polynomials, this means that $J$ should be symmetric, that is $J=-J$, and $a_{-\bar k}=\bar a_{\bar k}$, where, as usual, $\bar a_{\bar k}$ denotes the complex conjugate of $a_{\bar k}$.  Since the completion in the $\sup$-norm of the trigonometric polynomials 
$$
s(y)=\sum_{\bar k\in J}a_{\bar k}e^{i2\pi \bar k\cdot y},
$$
is exactly $C(\bbT^P)$, the assertion follows.

We henceforth denote by $\AP_*(\R^N)$ this subspace of $\AP(\R^N)$ and we will assume that  the noise functions $g_k$, $k\in\N$, belong to $\AP_*(\R^N)$.   

By a well known extension of the Stone-Weierstrass theorem (see, e.g,  \cite{DS}, p.274--276, Theorem~18 and Corollary~19) we have that $\AP_*(\R^N)$ is isometrically isomorphic with $C(\bbG_{*N})$, where $\bbG_{*N}$ is the  compactification  of $\R^N$ associated with
 $\AP_*(\R^N)$, which is
a topological group whose topology is generated by $\{e^{i2\pi k \l_j\cdot x}\,:\, j=1,2,\cdots,P, k\in\Z\}$, under which it is a compact topological space. 
We denote by $\B_*^1(\R^N)$ the completion of $\AP_*(\R^N)$ with respect to the semi-norm $N_1$ defined in \eqref{e2.N1}.  Therefore, we have that $\B_*^1(\R^N)$ is isometrically isomorphic with 
$L^1(\bbG_{*N})$. From the isometric embeddings $\AP_*(\R^N)\hookrightarrow \AP(\R^N)$ and
$\B_*^1(\R^N)\hookrightarrow \B^1(\R^N)$, there follow the isometric embeddings $C(\bbG_{*N})\hookrightarrow C(\bbG_N)$ and $L^1(\bbG_{*N})\hookrightarrow L^1(\bbG_N)$.

 For simplicity, let us first consider the situation where we have as the initial data $u_0$ in \eqref{e2.1}-\eqref{e2.2} a trigonometrical polynomial. So, for some finite symmetric set $J\subset\Z^P$ as above, with $a_{-\bar k}=\bar a_{\bar k}$,
   $u_0$ can be written as
\begin{equation}\label{e4.0}
u_0(x)=\sum_{\bar k\in J} a_{\bar k}e^{2\pi i\bar k\cdot y(x)}.
\end{equation}
Therefore, $u_0(x)=v_0(y(x))$ where
\begin{equation}\label{e4.0'}
v_0(y)=\sum_{\bar k\in J} a_{\bar k}e^{2\pi i \bar k\cdot y}
\end{equation}
also,  $g_k(x)=h_k(y(x))$, with $h_k\in C(\bbT^P)$, 
and, as defined above, $y(x)=(y_1(x),\cdots,y_P(x))$, with
\begin{equation}\label{e4.0''}
y_j(x)=\l_j\cdot x=\sum_{l=1}^n\l_{jl}x_l,\ \l_j=(\l_{j1},\cdots,\l_{jN}).
\end{equation}
Consider the equation
\begin{equation}\label{e4.1}
v_t+\div_y \tilde \bff(v)=\tilde \Phi\,d\tilde W,\quad v=v(t,y),\ \tilde \Phi\,d\tilde W=\sum_{k=1}^\infty h_k(y)\,d\b_k,
\end{equation}
with $\tilde \bff=(\tilde f_1,\cdots, \tilde f_P)$
$$
\tilde f_j(v)=\l_j\cdot \bff(v)=\sum_{l=1}^N\l_{jl} f_l(v),\quad j=1,\cdots,P.
$$
The initial value problem for \eqref{e4.1} is formed by prescribing the initial condition in $\bbT^P$:
\begin{equation}\label{e4.1'}
v(0,y)=v_0(y).
\end{equation}

By the results in \cite{DV,DV'} there is a unique kinetic solution 
$$
v(t,y)\in L^1(\Om; L^\infty((0,T); L^p(\bbT^P))),
$$ 
for all $p\ge1$, since $\bff$ is Lipschitz. Let $\tilde w=v-\tilde J$. 

\begin{definition}\label{D:5.1} Let $T>0$ be given and assume $v_0\in L^1(\bbT^P)$. A $L^1(\bbT^P)$-valued stochastic process adapted
to $\{\F_t\}$ is said to be a periodic entropy solution  of \eqref{e4.1}-\eqref{e4.1'} if, for almost all $\om\in\Om$,
\begin{enumerate}
 \item[(1P)] $u\in L^1(\bbT^P)$-weakly continuous on $[0,T]$,
 \item[(2P)] $u\in L^\infty([0,T]; L^1(\bbT^P))$,
 \item[(3P)] for all nonnegative $\tilde\varphi\in C_0^\infty ((-T,T)\X\bbT^P)$ and $\a\in\R$
 \begin{multline}\label{e4.3'}
 \int_0^T\int_{\bbT^P} \{(|\tilde w-\a|-|v_0-\a|)\tilde\varphi_t+\sgn(\tilde w-\a)(\tilde\bff(\tilde w+\tilde J)-\tilde\bff(\a+\tilde J))\cdot\nabla\tilde\varphi\\
 -\sgn(\tilde w-\a)\bff'(\a+\tilde J)\nabla\tilde  J\tilde\varphi\}\,dy\,dt\ge0,
 \end{multline}
 \end{enumerate}
 \end{definition}

The entropy solution of \eqref{e4.1}-\eqref{e4.1'} also satisfies the contraction property \eqref{e1.4P} and  it is  therefore unique. Also, it can be obtained as above by the vanishing viscosity method, in the same way, actually simpler, as it was done in the previous section for proving the existence of the 
$\BAP$-entropy solution. Therefore, in particular, it must coincide with the kinetic solution satisfying the definition in \cite{DV2}. 

We next establish a result which is the analogue of theorem~2.1 of \cite{Pv}, where the method of reduction to the periodic case was introduced.

\begin{theorem}\label{T:4.1} Let $v:\Om\X(0,T)\X\bbT^P\to\R$, be a periodic entropy solution of \eqref{e4.1}-\eqref{e4.1'}, where $v_0(y)$ is a trigonometric polynomial as in \eqref{e4.0'}. Let $y(x)$ be as in \eqref{e4.0''}. Then, there exists a set $Z\subset\R^P$ of total measure, that is,  $\R^P\setminus Z$ has $P$-dimensional Lebesgue measure zero, such that, for all  $z\in Z$, the function $u(t,x)=v(t,z+y(x))$ is an entropy $\BAP$-solution of an initial value problem as \eqref{e2.1}-\eqref{e2.2} with initial function $u_0(x)=v_0(z+y(x))$ and noise functions $h_k(z+y(x))$. Moreover, $Z$ does not depend on $\om\in\Om$ and can be taken as
the same for all trigonometric polynomials $v_0(y)$ in a countable family $\mathcal{T}$ dense in $L^1(\bbT^P)$. 
\end{theorem}

\begin{proof} Except for the independence of $Z$ with respect to $\om\in\Om$, the proof is totally similar to the one of theorem~2.1 of \cite{Pv}, and we refer to the latter for the proof of the first part. We assume that $\F$ has a countable  basis and let $\{\gamma_\ell(\om)={\bf1}_{A_\ell},:\, \ell \in\N\}$ where $\{A_\ell\,:\,\ell\in\N\}$ is  a basis for $\F$.  Also, let us assume that $v_0\in\mathcal{T}$ where 
$\mathcal{T}$ is a countable family of trigonometric polynomials dense in $L^1(\bbT^P)$.
Set $J(\om,z,t,x):=\sum_{k\in\N}h_k(z+y(x))\b_k(t,\om)$, $w(\om,z,t,x)=v(\om,t,z+y(x))-J(\om,z,t,x)$ and $u_0^z(x)=v_0(z+y(x))$.  Let $Z_\ell(v_0)\subset\R^P$ be the set of Lebesgue points $z\in\R^P$ of
\begin{multline}\label{e4.122}
I_\ell(v_0)=\int_\Om \gamma_\ell(\om) \int_0^T\int_{\R^N} \{(|w(\om,z,t,x)-\a|-|u_0^z-\a|)\varphi_t \\
 +\sgn(w(\om,z,t,x)-\a)(\bff(w(\om,z,t,x)+J(\om,z,x,t))-\bff(\a+J(\om,z,t,x)))\cdot\nabla\varphi \\
  -\sgn(w(\om,z,t,x)-\a)\bff'(\a+ J(\om,z,t,x))\nabla J(\om,z,t,x)\varphi\}\,dx\,dt\,d\bbP(\om),
\end{multline} 
where $\alpha\in\Q$ and $\varphi$ runs along a countable dense subset of $C_c^\infty([0,T)\X\R^N)$. We then define
$Z(v_0):=\bigcap_{\ell\in\N} Z_\ell(v_0)$, $Z=\bigcap_{v_0\in\mathcal{T}}Z(v_0)$. We can easily check that $Z$ satisfies the assertion of the theorem. 
 
\end{proof}

Together with Theorem~\ref{T:4.1} the following lemma is also a very important ingredient in  the method of reduction to the periodic case in \cite{Pv}. In the latter, the analogue of \eqref{e5.elem} below is derived from Birkhoff's ergodic theorem. Here we give a different proof which has the advantage to give the validity of the referred equation for all $z_0\in\R^P$.  

\begin{lemma}\label{L:5.elem} If $w\in L^1(\bbT^P)$, $z_0\in\R^P$, $y(x)=(\l_1\cdot x, \cdots,\l_P\cdot x)$,
$x\in\R^N$, then we may define the map $x\mapsto w(z_0+y(x))$ as a function in $\B_*^1(\R^N)$. Moreover, we have for the $\B^1$-norm of  this function
\begin{equation} \label{e5.elem}
\Medint_{\R^N}|w(z_0+y(x))|\,dx=\int_{\bbT^P}|w(y)|\,dy.
\end{equation}
In particular, the mapping $\mathfrak{Y}_{z_0}: w(y)\mapsto w(z_0+y(x))$ is an isometric isomorphism between $L^1(\bbT^P)$ and $\B_*^1(\R^N)$. 
\end{lemma}

\begin{proof} Consider the elementary trigonometric functions $E_0(y)=1$ and
$$
E_j^\pm(y)=e^{\pm i2\pi y_j}:[-1/2,1/2]^P\to\C,\quad j=1,\cdots,P,
$$ 
which can be viewed as functions on the $P$-dimensional torus $\bbT^P$, by the usual identification 
of $[-1/2,1/2]^P$ with periodic conditions on the boundary and the $P$-dimensional torus.  
We have $E_j^\pm(z_0+y(x))=e^{\pm i2\pi({z_0}_j+ \l_j\cdot x)}$ which clearly belong to $\AP(\R^N)$, since they are indeed periodic with period $(2\pi/(\l_j)_1,\cdots,2\pi/(\l_j)_N)$, $j=1,\cdots,P$. 
Since the (complex valued) continuous periodic functions on $[-1/2,1/2]^P$, or $C(\bbT^P)$,  form a closed algebra generated by the elementary trigonometric functions $E_j^\pm(y)$, $j=0,1,\cdots,P$, and $\AP(\R^N)$ is also a closed algebra, it follows that for any (complex valued) continuous periodic function $F\in C(\bbT^P)$, $F(z_0+y(x))\in\AP(\R^N)$. Observe also that we have, concerning the mean-value of $F(z_0+y(x))$,
$$
\Me(F):=\lim_{R\to\infty}\frac1{R^N}\int_{C_R}F(z_0+y(x))\,dx=\int_{\bbT^P}F(y)\,dy,
$$
since this is true when $F$ is a trigonometric polynomial, that is, when $F$ is a finite linear combination of $E_0(y)$ and trigonometric exponentials of the type $E^{\bar k}(y)=e^{i 2\pi \bar k\cdot y}$, with $\bar k=(k_1,\cdots,k_P)\in\Z^P$,  and these  are dense in 
$C(\bbT^P)$ with respect to the uniform topology. In particular, for any continuous periodic 
$F:\bbT^P\to\C$, the 
$\B^1$-norm of  $F(z_0+y(x))$ verifies
\begin{equation}\label{e5.elem'}
\Medint_{\R^N} |F(z_0+y(x))|\,dx=\int_{\bbT^P}|F(y)|\,dy.
\end{equation}
Since $C(\bbT^P)$ is dense in $L^1(\bbT^P)$, we deduce that, given $w\in L^1(\bbT^P)$, 
we can find a sequence $w_n\in C(\bbT^P)$, $n\in\N$, with $w_n\to w$ in $L^1(\bbT^P)$ and, so,
$w_n(z_0+y(x))$ is a Cauchy sequence in $\B^1(\R^N)$. Therefore, there exists a $g\in\B^1(\R^N)$ such that $w_n(z_0+y(x))\to g$ in $\B^1(\R^N)$. We notice that this function $g$ does not depend on the specific sequence of functions $w_n\in C(\bbT^P)$ converging to $w$ in $L^1(\bbT^P)$. Indeed,
if $\tilde w_n$ is another sequence in $C(\bbT^P)$ with $\tilde w_n\to w$ in $L^1(\bbT^P)$, then, by \eqref{e5.elem'},
 \begin{equation*}
\lim_{n\to\infty}\Medint_{\R^N} |w_n(z_0+y(x))-\tilde w_n(z_0+y(x))|\,dx=\lim_{n\to\infty}\int_{\bbT^P}|w_n(y)-\tilde w_n(y)|\,dy=0,
\end{equation*}
and so $w_n(z_0+y(x))$ and $\tilde w_n(z_0+y(x))$ converge to the same limit in $\B^1(\R^N)$. 
We may denote, without ambiguity, $g(x):= w(z_0+y(x))$. Moreover, since \eqref{e5.elem} holds for $w_n$, it also holds for $w$. 

Finally, concerning the fact that the mapping  ${\mathfrak Y}_{y_0}: w(y)\mapsto w(z_0+y(x))$ is an isometric isomorphism between $L^1(\bbT^P)$ and $\B_*^1(\R^N)$, we have the following. That this mapping is injective it is clear. The fact that it is onto follows from the fact that any $g\in \B_*^1(\R^N)$ may be approximated in $\B_*^1(\R^N)$ by trigonometric polynomials in $\AP_*(\R^N)$, $g^n(y(x))$ with $g^n\in C(\bbT^P)$ and $g^n$ converging in $L^1(\bbT^P)$ to some $w\in L^1(\bbT^P)$. This then proves that $g$ may be represented as $w(z_0+y(x))$, which implies that the mapping is onto.

\end{proof}

The following corollary is useful in connection with Theorem~\ref{T:4.1}. 

\begin{corollary} \label{C:newcor} Let $v:\Om\X(0,T)\X\bbT^P\to\R$, be the periodic entropy solution of \eqref{e4.1}-\eqref{e4.1'} with $v_0\in L^1(\bbT^P)$. Let $Z$ be the set of total measure given by Theorem~\ref{T:4.1}. Let $z\in Z$ be  fixed and  $y(x)=(\l_1\cdot x,\cdots, \l_P\cdot x)$. Then,  the function $u(t,x)=v(t,z+y(x))$ is a $\BAP$-entropy solution of \eqref{e2.1}-\eqref{e2.2} with initial function $v_0(z+y(x))$ and noise functions $g_k^z(x)=h_k(z+y(x))$. 
\end{corollary}

\begin{proof} Indeed, from the last lemma it follows, if $v_0^\a(y)$ is a sequence of trigonometric polynomials in $\mathcal{T}$ 
approximating $v_0(y)$ in $L^1(\bbT^P)$, then  
\begin{equation} \label{e5.elem''}
\Medint_{\R^N}|v_0^\a(z+y(x))-v_0(z+y(x))|\,dx=\int_{\bbT^P}|v_0^\a(y)-v_0(y)|\,dy,
\end{equation}
and so $v_0^\a(z+y(x))\to v_0(z+y(x))$ in $\B^1(\R^N)$ as $\a\to\infty$.
Therefore, if $u^{\a,z}(t,x)=v^\a(t,z+y(x))$ is the $\BAP$-entropy solution  of \eqref{e2.1}-\eqref{e2.2} with 
$u^{\a,z}(0,x)=v_0^\a(z+y(x))$, according to Theorem~\ref{T:4.1}, and $u^z(t,x)$ is the corresponding solution with 
 initial function $u(0,x)=v_0(z+y(x))$,   using \eqref{e2.4}, we obtain that $u^{\a,z}\to u^z$ in $L^\infty((0,T); \B^1(\R^N))$, as $\a\to\infty$, 
a.s.\ in $\Om$. Again, since by \eqref{e3.new} the norm of the $u^{\a,z}$'s in $L^1(\Om; L^\infty((0,T);\B^1(\R^N)))$  are
uniformly bounded by a function in $L^1(\Om)$, we conclude by dominated convergence that 
\begin{equation}\label{e.newcor}
u^{\a,z}\to u^z \quad \text {in $L^1(\Om;L^\infty((0,T); \B^1(\R^N)))$}.
\end{equation}
Finally, using again  Lemma~\ref{L:5.elem}, we deduce that we must have $u^z(t,x)=v(t, z+y(x))$, where $v(t,y)$ is the entropy solution 
of \eqref{e4.1}-\eqref{e4.1'}.
\end{proof}

\subsection{The limit as $z\to0$}\label{SS:4.1}  In this subsection we consider the limit as $z\to0$ of the 
$\BAP$-entropy solutions $u^z(t,x)=v(t, z+y(x))$ given by Corollary~\ref{C:newcor} and show that they converge to a $\BAP$-solution of \eqref{e2.1}-\eqref{e2.2}. Observe that, since 
$\B^1_*(\R^N)\subset \B^1(\R^N)$ such $\BAP$-entropy solutions belong to $\B_*^1(\R^N)$.
Similarly, if a $L^1(\bbG_{N})$-semigroup solution is the limit in $L^1(\Om; L^\infty((0,T); L^1(\bbG_N)))$ of  $\BAP$-entropy solutions of \eqref{e2.1}-\eqref{e2.2} belonging a.s.\ to $L^\infty((0,T); \B_*^1(\R^N))$, then, a.s., it belongs to  $L^\infty((0,T); L^1(\bbG_{*N}))$. We then, henceforth, call such $L^1(\bbG_N )$-semigroup solutions  
$L^1(\bbG_{*N})$-semigroup solutions of \eqref{e1.1}-\eqref{e1.2}. 

For the discussion in this subsection we assume the non-degeneracy condition \eqref{e5.NDC}, in Section~\ref{S:6}, to ensure the improved regularity of the periodic entropy solutions proved in \cite{DV2}.  
 
By Lemma~\ref{L:5.elem}, for any $z\in\R^P$,  the mapping ${\mathfrak Y}_z$, $v(y)\mapsto v(z+y(x))$ is an isometric isomorphism $L^1(\bbT^P)\to \B_*^1(\R^N)$. We denote ${\mathfrak Y}_0$ simply by ${\mathfrak Y}$.  
For $s\in\R$, $q\ge1$, let us define 
\begin{equation}\label{e.isometry0}
{\mathcal W}_{*}^{s,q}(\R^N):={\mathfrak Y} [W^{s,q}(\bbT^P)]
\end{equation}
 and
  \begin{equation}\label{e.isometry}
  \|v(y(\cdot))\|_{{\mathcal W}_{*}^{s,q}(\R^N)}=\|v\|_{W^{s,q}(\bbT^P)}.
  \end{equation}
Observe that the images by $\mathfrak{Y}_z$ of subspaces $W\subset L^1(\bbT^P)$ that are invariant by translations, 
that is,   $g(\cdot)\in W$ if and only if $g(z+\cdot)\in W$, for all $z\in\R^P$, are all the same subspace of 
$\B_*^1(\R^N)$. In particular,  ${\mathfrak Y}_z [W^{s,q}(\bbT^P)]={\mathfrak Y} [W^{s,q}(\bbT^P)]$ for all $z\in\R^P$.

\begin{theorem}\label{T:4.2}  Let $z_n\in Z$ be a sequence converging to 0 and let $u^n(t,x)=v(t,z_n+y(x))$ be the $\BAP$-entropy solution given by Corollary~\ref{C:newcor}, where 
$v(t,y)$ is the periodic entropy solution of \eqref{e4.1}-\eqref{e4.1'}, with initial function $v_0\in C(\bbT^P)$. Then, $u^n$ converges in $L^1(\Om;L^1((0,T); L_\loc^1\cap\B^1(\R^n)))$  to a $\BAP$-entropy solution of 
\eqref{e2.1}-\eqref{e2.2} with $u_0(x)=v_0(y(x))$, which then may be represented  as $u(t,x)=v(t,y(x))$. 

As a consequence, let $u_0\in L^1(\bbG_{*N})$, so that  $u_0(x)=v_0(y(x))$ for some $v_0\in L^1(\bbT^P)$.   Then $u(t,x)=v(t,y(x))$ is the $L^1(\bbG_{*N})$-semigroup solution of \eqref{e1.1}-\eqref{e1.2}  where 
$v(t,y)$ is the periodic entropy solution of \eqref{e4.1}-\eqref{e4.1'}, with initial function $v_0$. 
\end{theorem}
  
  \begin{proof} {\em Step \#1.}  The first part of the statement  is proved following the same steps as the proof of the existence of a  $\BAP$-entropy solution of \eqref{e2.1}-\eqref{e2.2} as the limit of a vanishing viscosity sequence of solutions to the parabolic approximation as it was done in Section~\ref{S:3}, with the following adaptations. Now, besides the sequence $u^n(t,x)$, we also consider the sequence $v^n(t,y):=v(t,z_n+y)$. Recall that $v^n(t,y)$ is the periodic entropy solution  of \eqref{e2.1}-\eqref{e2.2} with initial function $v_0(z_n+y)$ and noise functions $h_k(z_n+y)$, $k\in\N$. We can proceed with the above mentioned compactness method along the usual steps, Kolmogorov's continuity lemma, Prohorov's theorem, Skorokhod's representation theorem, etc., corresponding to Propositions~\ref{P:3.2}, \ref{P:3.2''}, \ref{P:3.2'''}, etc., simultaneously for both $u^n$ and $v^n$. While the steps for the sequence $u^n$ are similar to those for the vanishing viscosity sequence, the same is true for the sequence $v^n$. We combine both procedures transferring the regularity results for $v^n$ over to $u^n$ through the map ${\mathfrak Y}$. 
 
 {\em Step \#2.} Thus, combining the corresponding Proposition~\ref{P:3.2} for $u^n$ and $v^n$, we get $u^n\in C^\l([0,T]; W_\loc^{-1,r}\cap {\mathcal W}_*^{-1,r}(\R^N))$. Concerning the results corresponding to Proposition~\ref{P:3.2''} for both $u^n$ and $v^n$, they can be combined by defining  
$$
 {\mathcal X}:=L^1((0,T);L_\loc^1\cap\B_*^1(\R^N))
 \bigcap C([0,T]; W_\loc^{-1,r}\cap {\mathcal W}_*^{-1,r}(\R^N)).
 $$    
  In the proof of the tightness corresponding to Proposition~\ref{P:3.2''}, tranferring the regularity of $v^n$ to $u^n$, we can now define $K_R=K_R^u\cap K_R^v$, where $K_R^u$ is as $K_R$ in the proof of Proposition~\ref{P:3.2''} and 
  \begin{multline*}
  K_R^v:= \{u\in {\mathcal X}\,:\, \|u\|_{C^\l([0,T];{\mathcal W}_*^{-1,r}(\R^N))}\le R,
\,  \|u\|_{L^1((0,T);{\mathcal W}_*^{s,r}(\R^N))}\le R,\\
  \|u\|_{L^\infty((0,T);\B_*^2(\R^N))}\le R\}.
  \end{multline*}
  The procedures to prove the tightness of the laws of $u^n$ in ${\mathcal X}$ are then totally similar to those in the proof of Proposition~\ref{P:3.2''}. Then Proposition~\ref{P:3.2'''} and the subsequent content of Section~\ref{S:3} may be repeated with no change, and this way we conclude that the sequence $u^n$ converges in $L^1(\Om;L^1((0,T);L_\loc^1\cap \B_*^1(\R^N)))$ to the $\BAP$-entropy solution of \eqref{e2.1}-\eqref{e2.2}, with $u_0(x)=v_0(y(x))$, and by Lemma~\ref{L:5.elem} it may be represented as $u(t,x)=v(t,y(x))$.  Indeed, by Lemma~\ref{L:5.elem} we deduce that $v_0(z+y(x))\to v_0(y(x))$, as $z\to0$,  in $\B^1(\R^N)$. We also have  that $v_0(z+y(x))\to v_0(y(x))$, as $z\to0$,  in $L_\loc^1(\R^N)$, by the continuity of $v_0$.
   Moreover, using again Lemma~\ref{L:5.elem},  we have
  \begin{multline*}
 \bbE\int_0^T\!\!\!\Medint_{\R^N}|u^n(t,x)-v(t,y(x))|\,dx\,dt
 \\=\bbE\int_0^T\!\!\!\Medint_{\R^N}|v(t,z_n+y(x))-v(t,y(x))|\,dx\,dt\\
 =\bbE\int_0^T\int_{\bbT^P}|v(t,z_n+y)-v(t,y)|\,dy\,dt\to0,\quad\text{as $z_n\to0$},
 \end{multline*}
 where we also use the continuity of translations in $L^1(\bbT^P)$. Therefore, $u^n(t,x)\to v(t,y(x))$ 
 in $L^1(\Om\X[0,T]\X\bbG_N)$, and so $v(t,y(x))$ is the $\BAP$-entropy solution $u(t,x)$ of \eqref{e2.1}-\eqref{e2.2} with $u_0(x)=v_0(y(x))$.
 
 {\em Step \#3.} Concerning the final part of the statement, it is proved as follows. When $u_0\in \AP_*(\R^N)$, by Lemma~\ref{L:5.elem} and its proof, $u_0(x)=v_0(y(x))$, for some $v_0\in C(\bbT^P)$, and so by the first part of the statement, $u(t,x)=v(t,y(x))$ is a $\BAP$-entropy solution of \eqref{e2.1}-\eqref{e2.2}. On the other hand, if $u_0\in L^1(\bbG_{*N})$, by Lemma~\ref{L:5.elem}, $u_0(x)=v_0(y(x))$, for some $v_0\in L^1(\bbT^P)$, and, if $v_0^n\in C(\bbT^P)$ is a sequence of continuous functions on the torus converging to $v_0$ in $L^1(\bbT^P)$,
  then, as in the proof of Theorem~\ref{T:3new},  the $\BAP$-entropy solutions with initial functions $u_0^n(x)=v_0^n(y(x))$, $u^n(t,x)=v^n(t,y(x))$,  converge in $L^1(\Om;L^\infty((0,T);L^1(\bbG_{*N})))$ to a 
  $L^1(\bbG_{*N})$-semigroup solution of \eqref{e1.1}-\eqref{e1.2}, which can be represented as 
  $u(t,x)=v(t,y(x))$.  
       
  \end{proof}

 \section{Asymptotic Behavior}\label{S:6}

In this section we study the asymptotic behavior of the $L^1(\bbG_{*N})$-semigroup solution obtained in the last section.  Thus, we keep considering  the algebra
generated by $\{e^{\pm 2\pi i\l_\ell\cdot x}\,:\, \ell=0,1,2,\cdots, P\}$, with $\l_\ell\in\R^N$, $\l_0=0$,   where $\Lambda=\{\l_1, \l_2, \cdots,\l_P\}$ is a $\Z$-linearly independent set in $\R^N$, and we keep  denoting the closure of  this algebra  in 
the $\sup$-norm by $\AP_*(\R^N)$. For any $g\in\AP_*(\R^N)$, we have $\Sp(g)\in\G_\Lambda$, where the latter is the smallest additive group containing $\Lambda$. We also keep assuming, as in the last section,  that the noise functions satisfy $g_k\in\AP_*(\R^N)$, $k\in\N$. For $y(x)=(\l_1\cdot x,\cdots, \l_P\cdot x)$ we have that $g_k(x)=h_k(y(x))$, where $h_k(y)\in C(\bbT^P)$, $k\in\N$.

In what follows we prove the Theorem~\ref{T:1.1} stated in Section~\ref{S:1'}. 

\begin{proof}[Proof of Theorem~\ref{T:1.1}]  We start by establishing the existence of an invariant measure for \eqref{e1.1}. From \eqref{e4.10} we can define   the transition semigroup in $L^1(\bbG_{*N})$ associated with \eqref{e1.1}:
$$
P_t\phi(u_0)=\bbE(\phi(u(t))),\quad \phi\in \mathcal{B}_b(L^1(\bbG_{*N})),
$$
where $u(t)$ denotes the $L^1(\bbG_{*N})$-semigroup solution with initial data $u_0$ at time $t$, and, for a metric space $E$, $\B_b(E)$ is the space of bounded Borel functions on $E$.

A probability measure $\mu$ on $L^1(\bbG_{*N})$ is a said to be an invariant measure for $(P_t)$ if we have
\begin{equation*}
P_t^*\mu =\mu,\quad t\ge0,\quad
\text{where $\la P_t^*\mu, \phi\ra=\la\mu, P_t\phi\ra$, for all $\phi\in \mathcal{B}_b(L^1(\bbG_{*N}))$.}
\end{equation*}
It can be easily checked that $P_t(u_0,\Gamma):=P_t\chi_\Gamma(u_0)$, $u_0\in L^1(\bbG_{*N})$, 
$\Gamma\in\mathcal{E}:=\mathbb{B}(L^1(\bbG_{*N}))$, the $\s$-algebra of Borel sets in $L^1(\bbG_{*N})$,  defines a Markovian transition function.

 Recall \eqref{e.isometry0} and \eqref{e.isometry} and let  $S\subset L^1(\bbG_{*N})$ be defined by
$$
S=\{u\in {\mathcal W}_*^{s,q}(\R^N)\,:\, \int_{\bbG_{*N}}u(x)\,dx=0\},
$$
where, as before, ${\mathcal  W}_*^{s,q}(\bbT^P)= \mathfrak{Y}[W^{s,q}(\bbT^P)]$ and $W^{s,q}(\bbT^P)$ is the Sobolev space such that the kinetic periodic solutions obtained in \cite{DV2} with initial data in $L^3(\bbT^P)$ a.s.\ belong to $L^1((0,T); W^{s,q}(\bbT))$, for all $T>0$.
More specifically, we also recall the decisive estimate  (42) from \cite{DV2},  
  for the kinetic periodic solution on
 $\bbT^P$,
 \begin{equation}\label{e5.DV}
 \bbE\|v\|_{L^1((0,T);W^{s,q}(\bbT^P))}
 \le k_0(\bbE\|v_0\|_{L^3(\bbT^P)}^3+1+T),
 \end{equation}
 for some $q>1$, where $k_0$ depends only on the data of the periodic problem, provided the non-degeneracy condition \eqref{e5.NDC3}, with \eqref{e5.NDC2}, stated below, holds.
We then define, 
$$
\|u\|_S:=\|u\|_{{\mathcal W}_*^{s,q}(\R^N)}.
$$
We notice that $S$ is a subspace of $L^1(\bbG_{*N})$ and $\|\cdot\|_S$ is a norm. Indeed,  since $W^{s,p}(\bbT^P)$ is continuously embedded in $L^1(\bbT^P)$, we have that if $\|u\|_S=0$, then $v=\mathfrak{Y}(u)=0$ in $L^1(\bbT^P)$, which, in turn,  by Lemma~\ref{L:5.elem}, implies that $u=0$ in $L^1(\bbG_{*N})$. The other properties for a norm are obviously checked. Thus, $\|\cdot\|_S$ is a norm in $S$.

Let $S_R:=\{u\in S\,:\, \|u\|_S\le R\}$. We claim that $S_R$ is compact in $L^1(\bbG_{*N})$. Indeed,
given a sequence $u_\a\in S_R$, we can find $v_\a(y(x))$, with $\|v_\a\|_{W^{s,p}(\bbT^P)}\le R$ 
and $u_\a=v_\a(y(x))$ in $L^1(\bbG_{*N})$. By the compactness of the embedding 
$W^{s,p}(\bbT^P)\subset L^1(\bbT^N)$, we may find a subsequence $v_{\a_k}$ converging in $L^1(\bbT^P)$ to some $v\in W^{s,p}(\bbT^P)$, with $\|v\|_{W^{s,p}(\bbT^P)}\le R$. Then,  by Lemma~\ref{L:5.elem}, $u_{\a_k}(x)=v_{\a_k}(y(x))$ converges in $L^1(\bbG_{*N})$  to  certain $u\in S_R$, which proves the compactness of $S_R$.

For a given trigonometrical polynomial $v_0$ with mean-value zero, let $u_0(x)=v_0(y(x))$, and let us define the probability measures 
$$
\la \mu_T,\phi \ra :=\frac1T\int_0^T P_t\phi (u_0) dt =\frac1T\int_0^T \la P_t^*\d_{u_0},\phi\ra dt. \quad \phi\in \B_b(L^1(\bbG_{*N})),
$$
We are going to show that the family  $\{\mu_T\}_{T>0}$  of measures over $\B_b(L^1(\bbG_{*N}))$ is tight, aiming to apply Prohorov's theorem (see \cite{Bi}).  Assume that
 $\phi\in \B_b(L^1(\bbG_{*N}))$ has support in $S_R^c=L^1(\bbG_{*N})\setminus S_R$, where $S_R$ is as above. 
 Thus, we have
 \begin{multline*}
 |\la  \mu_T,\phi\ra|=\left| \frac1T\int_0^T P_t\phi (u_0) dt\right|=\left| \frac1T\int_0^T \bbE\phi (u(t)) dt\right|\le \|\phi\|_\infty\frac1{R} \frac1T
 \bbE\int_0^T \|u(t)\|_S\,dt\\
 = \|\phi\|_\infty\frac1{R}\frac1T \bbE\int_0^T \|v(t)\|_{W^{s,q}(\bbT^P)}\,dt
 \le \frac{C\|\phi\|_\infty}{R},
 \end{multline*} 
for some constant $C>0$ where we have used  \eqref{e5.DV}, which proves the desired tightness of $ \mu_T$, $T>0$. So, by Prohorov's theorem, $\mu_T$ is relatively compact in the space of probability measures and so there is a subsequence $\mu_{T_k}$ and $\mu\in \M_1(L^1(\bbG_{*N}))$ such that $\mu_{T_k}\wto \mu$. 

 Let us prove that $\mu$ is an invariant measure for \eqref{e1.1}. Indeed, for any $t\ge0$ and $\phi\in \B_b(L^1(\bbG_{*N}))$, we have
 \begin{multline*}
 \la P_t^* \mu,\phi\ra =\la  \mu, P_t\phi\ra =\lim_{n\to\infty}\la \mu_{T_n}, P_t\phi\ra\\
 =\lim_{n\to\infty}\frac1{T_n}\int_0^{T_n}P_s (P_t\phi)(u_0)\,ds= \lim_{n\to\infty}\frac1{T_n}\int_0^{T_n}
 P_{t+s}\phi (u_0) \,ds\\
 =\lim_{n\to\infty}\frac1{T_n}\int_t^{T_n+t} P_s\phi(u_0)\,ds=  \lim_{n\to\infty}\frac1{T_n}\int_0^{T_n} P_s\phi(u_0)\,ds\\
 -\lim_{n\to\infty}\frac1{T_n}\int_0^{t} P_s\phi (u_0)\,ds+\lim_{n\to\infty}\frac1{T_n}\int_{T_n}^{T_n+t} P_s\phi (u_0)\,ds\\
 = \lim_{n\to\infty}\frac1{T_n}\int_0^{T_n} P_s\phi(u_0)\,ds=\la\mu,\phi\ra,
 \end{multline*}
 which proves that $\mu$ is an invariant measure for \eqref{e1.1}.


 We are now going to prove  the uniqueness of the invariant measure $\mu$ for the problem \eqref{e1.1}. We again need a non-degeneracy condition on the flux function $\bff(u)$ as in \cite{DV2}. Nevertheless here we need to refine that condition as follows.  Let $\G_{\Lambda}$ be the additive group generated by 
 $\Lambda=\{\l_1,\l_2, \cdots, \l_P\}$.  For $\b\in\G_{\Lambda}$, $\b=n_1\l_1+n_2\l_2+\cdots+n_P\l_P$, $n_j\in\Z$.   We denote ${\bf n}_\b:=(n_1,\cdots,n_P)$ and $|{\bf n}_\b|:=\sqrt{n_1^2+\cdots+n_{P}^2}$. 
   
 Let $\abf:=\bff'$  and define $\iota:(0,\infty)\X\N\to(0,\infty)$ by
 \begin{equation}\label{e5.NDC1}
 \iota(\d,J)=\sup_{\tau\in\R,\b\in \G_{\Lambda}, |{\bf n}_\b|\sim J}
 |\{\xi\in\R\,:\, |\tau+\abf(\xi)\cdot\b|\le \d\}|.
 \end{equation}
 We assume that 
 \begin{equation}\label{e5.NDC}
 \iota(\d,J)\le c_1\left(\frac{\d}{J}\right)^\theta,
 \end{equation}
 for some $c_1>0$ and $0<\theta<1$. 
  
 Observe that in the case where the $u_0(x)=v_0(y(x))$  is a trigonometric polynomial and $v$ is the entropy solution of 
 the corresponding periodic problem  as in Theorem~\ref{T:4.1}, then
 the flux function $\tilde \bff(v)=(\l_1\cdot \bff(v),\cdots,\l_P\cdot \bff(v))$ 
  satisfies the condition, 
 \begin{equation}\label{e5.NDC3}
 \tilde\iota(\d,J)\le c_1\left(\frac{\d}{J}\right)^\theta,
 \end{equation} 
 for all $\d>0$, for some $c_1>0$ and $0<\theta<1$, where
 \begin{equation}\label{e5.NDC2} 
\tilde\iota(\d,J):=\sup_{\tau\in\R,{\bf n}\in \Z^P, |{\bf n}|\sim J}
 |\{\xi\in\R\,\,|\tau+ \tilde \abf(\xi)\cdot {\bf n}|\le \d\}|
 \end{equation}
  with $\tilde \abf=\tilde \bff'$.

 Condition \eqref{e5.NDC3}, with \eqref{e5.NDC2}, is indeed  what is needed in order  for the analysis in \cite{DV2} to be carried out and so the results in \cite{DV2} hold for this apparently weaker condition.

Uniqueness of the invariant measure $\mu$ for \eqref{e1.1} follows from the fact that the periodic entropy solutions of \eqref{e4.1}-\eqref{e4.1'} are  kinetic solutions in the sense of \cite{DV2} and so,
any two of these solutions $v^1, v^2$, with initial data $v_0^1, v_0^2\in L^3(\bbT^P)$, with  mean-value zero, respectively, satisfy (by the last equation of section~4 in \cite{DV2})
\begin{equation}\label{e5.8}
\lim_{t\to\infty}\|v^1(t)-v^2(t)\|_{L^1(\bbT^P)}=0,\, a.s.
\end{equation}
{}From \eqref{e5.8} we obtain, for any two  $L^1(\bbG_{*N})$-entropy solutions of \eqref{e2.1}, with trigonometric polynomials $u_0^i(x)=v_0^i(y(x))$, $i=1,2$,  with zero mean-value, as initial data and  noise coefficients $g_k(x)=h_k(y(x))$, $k\in\N$, the equation   
\begin{equation}\label{e5.9}
\lim_{t\to\infty}\|u^1(t)-u^2(t)\|_{L^1(\bbG_N)}=0,\, a.s.
\end{equation}
 and this, together with \eqref{e2.4},  implies the uniqueness of the invariant measure
 $\mu$ for \eqref{e1.1}.  Indeed, let $\nu$ be an invariant measure for \eqref{e1.1}, and  $P_t$ the corresponding transition semigroup as above, that is, $P_t^*\nu=\nu$.  Given $\phi\in \Lip(L^1(\bbG_{*N}))$ let $u_0(\cdot)=v_0(y(\cdot))$, for a trigonometrical polynomial $v_0(y)$ with mean-value zero. Let $\mu$ be the invariant measure constructed above, that is $\mu=\lim_{T_n\to\infty} \mu_{T_n}$, where $\mu_{T}=\frac1T\int_0^T P_t^*\d_{u_0}\,dt$. We have 
 \begin{multline*}
 |\la\nu,\phi\ra-\la\mu_{T_n},\phi\ra|=\frac1{T_n}\int_0^{T_n}|\la P_t^*\nu,\phi\ra-\la P_t^*\d_{u_0},\phi\ra|\,dt\\
= \frac1{T_n}\int_0^{T_n}\left| \int_{L^1(\bbG_{*N})} (P_t\phi( v_0)-P_t\phi(u_0))\,d\nu(v_0) \right|\,dt\\
 \le  \frac1{T_n}\int_0^{T_n}\int_{L^1(\bbG_{*N})} C_\phi\, \bbE\| v(\cdot,t)- u(\cdot,t)\|_{L^1(\bbG_{*N})}\,d\nu(v_0) \,dt\\
\le \int_{L^1(\bbG_{*N})}  C_\phi \,\bbE\frac1{T_n}\int_0^{T_n} \| v(\cdot,t)-u(\cdot,t)\|_{L^1(\bbG_{*N})}\,dt\,d\nu(v_0)\to 0, 
\end{multline*}
where $v(\cdot,t),u(\cdot,t)$ are the $L^1(\bbG_{*N})$-entropy solutions associated with the initial data $v_0,u_0$, respectively. 
Hence, making $T_n\to\infty$, using \eqref{e5.9},  we conclude
$$
 |\la\nu,\phi\ra-\la \mu, \phi\ra|=0,
 $$
 and so
 \begin{equation}\label{e50.uni}
 \la\nu,\phi\ra=\la \mu, \phi\ra
 \end{equation}
 for all $\phi\in\Lip(L^1(\bbG_{*N}))$. Now, it is easy to extend \eqref{e50.uni} to all 
 $\phi\in\B_b(L^1(\bbG_{*N}))$. Indeed,  we first extend it to $\phi={\bf1}_F$, where $F$ is any closed subset of $L^1(\bbG_{*N})$, by using a sequence of Lipschitz continuous functions, $\phi^\ve$, converging everywhere to ${\bf 1}_F$, e.g.,  $\phi^\ve(\cdot)=\max\{0,1-\frac1\ve\dist(\cdot,F)\}$.  Then, by the regularity of the probability measures $\mu$ and $\nu$,   we extend \eqref{e50.uni} to 
 $\phi={\bf 1}_A$,  for any Borel set $A$, that is,  $\nu(A)=\mu(A)$ for all Borel sets of $L^1(\bbG_{*N})$, which  implies the uniqueness of the invariant measure  $\mu$ for \eqref{e1.1}. 
 
 \end{proof}

\end{document}